%% file: Homogeneous_bands.tex
\documentclass[12pt,a4]{amsart}
\pdfoutput=1
\usepackage{a4wide}
\usepackage{slashbox}
\usepackage[all,cmtip]{xy}
\usepackage{amssymb}
\usepackage{enumitem}   
\usepackage{colortbl}
\usepackage{hyperref}
\usepackage[font=small,labelfont=bf]{caption} 
\usepackage{tkz-graph}
\tikzset{node distance=2cm, auto}
\usepackage{pdfpages}

\newtheorem{theorem}{Theorem}[section]
\newtheorem{lemma}[theorem]{Lemma}

\newtheorem{proposition}[theorem]{Proposition}
\newtheorem{corollary}[theorem]{Corollary}

\theoremstyle{definition}
\newtheorem{example}[theorem]{Example}
\newtheorem{definition}[theorem]{Definition}
\newtheorem{remark}[theorem]{Remark} 

\begin{document}

\title[Homogeneous bands	]%
{Homogeneous bands}



\author{Thomas Quinn-Gregson}
\email{tdqg500@york.ac.uk}
\address{Department of Mathematics\\University
  of York\\York YO10 5DD\\UK}





\thanks{This work forms part of my Phd at the University of York, supervised by Professor Victoria Gould, and funded by EPSRC}


\begin{abstract} 
A band $B$ is a semigroup such that $e^2=e$ for all $e\in B$. A countable band is called homogeneous if every isomorphism between finitely generated subbands extends to an automorphism of the band. In this paper, we give a complete classification of all the homogeneous bands. We prove that a homogeneous band belongs to the variety of regular bands, and has a  homogeneous structure semilattice. 
\end{abstract}


\maketitle

\section{Introduction}\label{sec:intro}

A \textit{structure} is a set $M$ together with a collection of finitary operations and relations defined on $M$. A countable structure $M$ is \textit{homogeneous} if any isomorphism between finitely generated (f.g.) substructures extends to an automorphism of $M$. Much of the model theoretic interest in homogeneous structures has been due to their strong connections to $\aleph_0$-categoricity and quantifier elimination in the context of \textit{uniformly locally finite} (ULF) structures (see, for example, \cite[Theorem 6.4.1]{Hodges97}). A structure $M$ is  ULF if there exists a function $f:\mathbb{N}\rightarrow \mathbb{N}$ such that for every substructure $N$ of $M$, if $N$ has a generating set of cardinality at most $n$, then $N$ has cardinality at most $f(n)$. If a ULF structure $M$ has finitely many operations and relations, then the property of homogeneity is equivalent to $M$ being both $\aleph_0$-categorical and having quantifier elimination. It is worth noting that any structure $M$ with no operations (known as \textit{combinatorial}) is ULF, since any subset $X$ is a substructure under the restriction of the relations on $M$ to $X$. 
  
 Consequently, homogeneity of mainly combinatorial structures has been studied by several authors, and complete classifications have been obtained for a number of structures including graphs in \cite{Lachlan80}, partially ordered sets in \cite{Schmerl79} and (lower) semilattices in \cite{Droste85} and \cite{Truss99}. There has also been significant progress in the classification of homogeneous groups and rings (see, for example, \cite{Cherlin91}, \cite{Cherlin93} and \cite{Saracino84}), and a complete classification of homogeneous finite groups is known (\cite{Cherlin2000}, \cite{Li99}). However, there are a number of fundamental algebraic structures which have yet to be explored from the point of view of homogeneity: in particular semigroups.
 
In this paper we consider the homogeneity of bands, where a band is a semigroup consisting entirely of idempotents. Our main result is to give a complete description of homogeneous bands. We show that every homogeneous band lies in the variety of  \textit{regular} bands and examine how our results fit in with known classifications, in particular showing that the structure semilattice of a homogeneous band is itself homogeneous. We may thus think of classification of homogeneous bands as an extension of the classification of homogeneous semilattices. All structures will be assumed to be countable.  

It follows from the work in \cite{Lean54} that bands are ULF, and so we need only to consider isomorphisms between finite subbands. Moreover, as each subsemigroup of a band is a band, we shall write `homogeneous bands' to mean homogeneous as a semigroup, without ambiguity. 

In Section 2 we outline the basic theory of bands, and a number of varieties of bands are described. All bands are shown to be built from a semilattice and a collection of simple bands, and the homogeneity of these two building blocks of all bands are completely described. In Section 3 a stronger form of homogeneity is considered and used to obtain a number of results on the homogeneity of regular bands. The structure of a general homogeneous band is also considered, and the results are used in Section 4 to obtain a description of all homogeneous normal bands. In Section 5 we consider the homogeneity of linearly ordered bands, that is, bands with structure semilattice a chain, and these are fully classified. Finally, in Section 6 all other bands are studied and are shown to yield no new homogeneous bands.

\section{The theory of bands}\label{sec:theory}

Much of the early work on bands was to determine their lattice of varieties:  a feat that was independently completed by Biryukov, Fennemore and Gerhard. In addition, Fennemore (\cite{Fennemore70}) determined all identities on bands, showing that every variety of bands can be defined by a single identity. The lower part of the lattice of varieties of bands contains the following vital varieties for this paper (see Figure \ref{lattice bands}): 
\begin{align*}
& \mathcal{LZ}:\text{left zero bands: } xy=x \text{ and }   \mathcal{RZ}:\text{right zero bands: } xy=y,\\
& \mathcal{RB}=\mathcal{LZ}\lor \mathcal{RZ}: \text{rectangular bands: } xyx=x, \\
& \mathcal{SL}: \text{semilattices: } xy=yx, \\
& \mathcal{LN}:\text{left normal bands: } xyz=xzy \text{ and } \mathcal{RN}: \text{right normal bands: } xyz=yxz, \\
& \mathcal{N}=\mathcal{LN}\lor \mathcal{RN}: \text{normal bands: } xyzx=xzyx, \\
& \mathcal{LG}: \text{left regular bands: } xyx = xy \text{ and } \mathcal{RG}: \text{right regular bands: } xyx = yx, \\
& \mathcal{G}=\mathcal{LG}\lor \mathcal{RG}: \text{regular bands: } zxzyz = zxyz,
\end{align*}
where the given relation characterizes the variety in the variety of bands. 

\begin{figure}
\begin{center}
\includegraphics[page=1,scale=0.6]{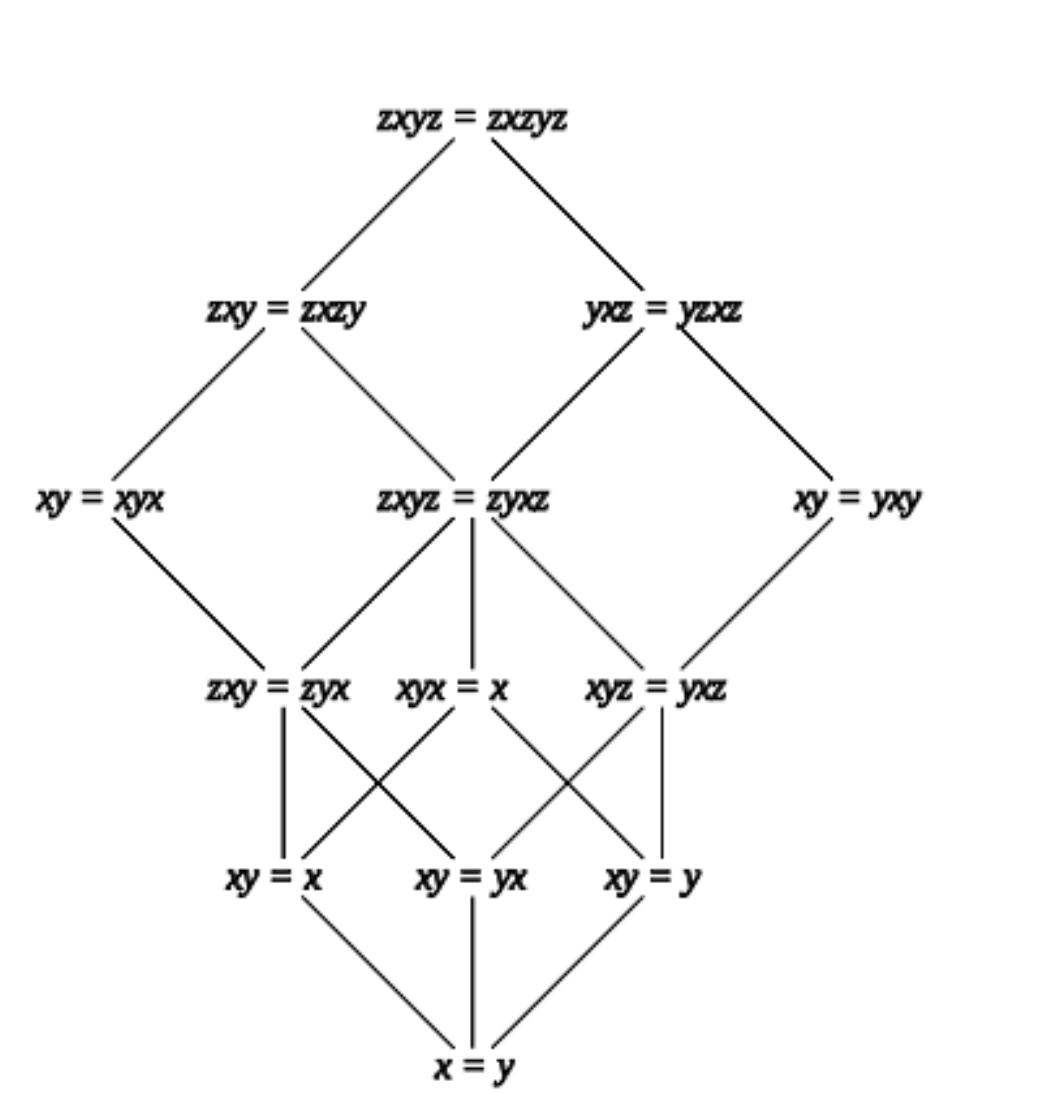}
\captionof{figure}{Lower part of the lattice of varieties of bands}
\label{lattice bands} 
\end{center}
\end{figure}
We shall proceed to give alternative descriptions of these varieties. We first consider the ideal structure on a band via its \textit{Green's relations}. Define a pair of quasi orders $\leq_r$ and $\leq_l$ on a band $B$ by 
\[   e \, \leq_r \, f \Leftrightarrow fe=e, \quad  e \, \leq_l \, f \Leftrightarrow ef=e.
\]
The relations $\leq_r$ and $\leq_l$ are called the \textit{Greens right and left quasi orders}, respectively. Note that $\leq_r$ is \textit{left compatible with multiplication on $B$}, that is, if $e\leq_r f$ and $g\in B$ then $ge \leq_r gf$; dually for $\leq_l$. Green's relations on $B$ are the equivalence relations given by
\begin{align*}
&  e \, \mathcal{R} \, f \Leftrightarrow [e\leq_r f, f \leq_r e] \Leftrightarrow [ef=f, fe=e];  \\
& e \, \mathcal{L} \, f \Leftrightarrow [e\leq_l f, f \leq_l e]  \Leftrightarrow [ef=e, fe=f]; \\
& e \, \mathcal{D} \, f \Leftrightarrow e \, \mathcal{J} \, f \Leftrightarrow [efe=e, fef=f]. 
\end{align*}
Every band comes equipped with a partial order $\leq$, called the \textit{natural order on $B$}, given by 
\[ e\leq f \Leftrightarrow ef=fe=e. 
\] 
Note that the natural order is preserved under morphisms, that is, if $\phi:B\rightarrow B'$ is a morphism between bands $B$ and $B'$  then for $e,f\in B$, 
\[ e\leq f \Rightarrow ef=fe=e \Rightarrow e\phi f \phi=f \phi e\phi =e\phi \Rightarrow e\phi\leq f\phi. 
\] 
\noindent {\bf Notation} Let $(P,\leq)$ be a poset with subsets $M,N$. Where convenient we denote the property that $m\geq n$ for all $m\in M$, $n\in N$ as $M\geq N$, and similarly for $M>N$. If $M=\{m\}$ then we may simplify this as $m>N$, and similarly for $M>n$. Given $p,q\in P$, if $p\not\geq q$ and $q\not\geq p$ then we say that \textit{$p$ and $q$ are incomparable}, denoted $p\perp q$. 

Given a band $B$ and a subset $A$ of $B$, we denote $\langle A \rangle$ as the subband generated by $A$, noting that $\langle A \rangle$ is the intersection of all subbands of $B$ which contain $A$. 

\begin{lemma} \label{less compatible} Let $B$ be a band with subsets $M$ and $N$ such that $M\geq N$ (under the natural order on $B$). Then $\langle M \rangle \geq \langle N \rangle$.
\end{lemma} 

\begin{proof}
Let $e_1,e_2,\dots ,e_r\in M$ and $f_1,f_2,\dots,f_s\in N$. Then as $e_k>\{f_1,f_s\}$ for all $1\leq k \leq r$, 
\[ (e_1e_2\cdots e_r)(f_1f_2\cdots f_s) =(e_1e_2\cdots e_{r-1})(f_1f_2\cdots f_s)=\cdots =f_1f_2\cdots f_s
\] 
and similarly $(f_1f_2\cdots f_s)(e_1e_2\cdots e_r)=f_1f_2\cdots f_s$. 
\end{proof}

Henceforth, a relation $\leq$ defined on a band will be assumed to be the natural order, unless stated otherwise. Recall that a \textit{semilattice} is a commutative band. A \textit{lower semilattice} is a poset in which the meet (denoted $\wedge$) of any pair of elements exists.

Let $(Y, \leq)$ be a lower semilattice; we may then consider the structure $(Y,\leq,\wedge)$, that is, a set with a single binary operation and relation. Homogeneity of lower semilattices was first considered by Droste in \cite{Droste85} and, together with Truss and Kuske, in \cite{Truss99}. Note that both articles consider the homogeneity of the corresponding structure $(Y,\leq,\wedge)$ (wherein \cite{Droste85} a lower semilattice is also considered simply as a poset). Since any morphism of the corresponding (semigroup) semilattice $(Y,\wedge)$ preserves $\leq$, their work effectively considers homogeneity of (semigroup) semilattices.

 With this in mind, for the rest of the paper, we refer simply to {\em semilattices} and {\em homogeneous semilattices}. Note that by Schmerl's classification of homogeneous posets in \cite{Schmerl79}, a semilattice is homogeneous as a poset if and only if it is isomorphic to $(\mathbb{Q},\leq)$.

A semilattice $Y$ is called a \textit{semilinear order} if $(Y,\leq)$ is non-linear and, for all $\alpha\in Y$, the set $\{\beta\in Y:\beta\leq \alpha\}$ is linearly ordered. This is equivalent to $Y$ not containing a \textit{diamond}, that is, distinct $\delta,\alpha,\gamma,\beta\in Y$ such that $\delta>\{\alpha,\gamma\}>\beta$ and $\alpha \perp \gamma$ with $\alpha\gamma=\beta$. There exists a unique (up to isomorphism) homogeneous semilattice which embeds all finite semilattices, called the \textit{universal semilattice}. 

\begin{proposition}[{{\cite{Droste85, Truss99}}}]  A non-trivial homogeneous semilattice is isomorphic to either $(\mathbb{Q},\leq)$, the universal semilattice or a semilinear order. 
\end{proposition} 

The second variety of bands required for the construction of an arbitrary band are \textit{rectangular bands}, that is, bands satisfying the identity $xyx=x$. Note that a band is rectangular if and only if it contains a single $\mathcal{D}$-class (and is thus \textit{simple}). Similarly \textit{left zero bands} are precisely the bands with a single $\mathcal{L}$-class (dually for \textit{right zero bands}).

\begin{proposition}[{{\cite[Theorem 1.1.3]{Howie94}}}] Let $L$ be a left zero semigroup and $R$ a right zero semigroup. Then $B_{L,R}=(L \times R, \cdot)$ forms a rectangular band, with operation given by
\[    (i, j) \cdot (k, \ell) = (i, \ell). 
\] 
Conversely, every rectangular band is isomorphic to some $B_{L,R}$. Greens relations on $B_{L,R}$ simplify to 
\[ (i,j) \, \mathcal{R} \, (k, \ell) \Leftrightarrow i=k \text{ and } (i,j) \, \mathcal{L} \, (k, \ell) \Leftrightarrow j=\ell. 
\]
\end{proposition} 

An isomorphism theorem for rectangular bands follows immediately from \cite[Corollary 4.4.3]{Howie94}, and is stated below: 

\begin{proposition}\label{rec bands iso} A pair of rectangular bands $B_{L,R}$ and $B_{L',R'}$ are isomorphic if and only if $|L|=|L'|$ and $|R|=|R'|$. Moreover, if $\phi_L:L\rightarrow L'$ and $\phi_R:R \rightarrow R'$ are a pair of bijections, then the map $\phi:B_{L,R} \rightarrow B_{L',R'}$ defined by 
\[ (i,j)\phi=(i \phi_L,j\phi_R)
\] 
is an isomorphism. Every isomorphism from $B_{L,R}$ to $B_{L',R'}$ can be constructed in this way, and may be denoted $ \phi=\phi_L \times \phi_R$. 
\end{proposition} 

For each $n,m\in \mathbb{N^*}=\mathbb{N}\cup\{\aleph_0\}$, we may thus denote $B_{n,m}$ to be the unique (up to isomorphism) rectangular band with $n$ $\mathcal{R}$-classes and $m$ $\mathcal{L}$-classes. 

\begin{corollary}\label{rec bands homog} Let $B_i=L_i\times R_i$ $(i=1,2)$ be a pair of isomorphic rectangular bands, with subbands $A_i$ ($i=1,2$). Then any isomorphism from $A_1$ to $A_2$ can be extended to an isomorphism from $B_1$ to $B_2$. In particular, rectangular bands are homogeneous. 
\end{corollary} 

\begin{proof} Since the class of rectangular bands forms a variety, each $A_i$ is a rectangular band and $A_i=L_i'\times R_i'$ for some $L_i'\subseteq L_i$, $R_i'\subseteq R_i$. Let $\theta=\theta_{L'}\times \theta_{R'}:A_1\rightarrow A_2$ be an isomorphism. Extend the bijection $\theta_{L'}:L_1'\rightarrow L_2'$ to a bijection $\theta_L:L_1\rightarrow L_2$, and similarly obtain $\theta_{R}$. Then $\hat{\theta}=\theta_L\times \theta_R$ extends $\theta$ as required. Taking $B_1=B_2$ gives the final result. 
\end{proof}

A structure theorem for bands was obtained by McLean in \cite{Lean54}, and can be stated as follows:  

\begin{proposition}\label{structure bands} Let $B$ be an arbitrary band. Then $\mathcal{D}$ is a congruence on $B$ and is such that $Y=S/\mathcal{D}$ is a semilattice and the $\mathcal{D}$-classes are rectangular bands. In particular $B$ is a semilattice of rectangular bands $B_{\alpha}$ (which are the $\mathcal{D}$-classes), that is,  
\[ B=\bigcup_{\alpha\in Y} B_{\alpha} \quad  \text{with} \quad   B_{\alpha}B_{\beta}\subseteq B_{\alpha \beta}.
\] 
\end{proposition}

The semilattice $Y$ is called the \textit{structure semilattice} of $B$. It follows from the proposition above that we understand the global structure of an arbitrary band, but not necessarily the local structure. We now consider a well-known construction to obtain bands with an explicit operation (see \cite[p.22]{Petrich99}, for example).  

Let $Y$ be a semilattice. To each $\alpha \in Y$ associate a rectangular band $B_{\alpha}$ and assume that $B_{\alpha} \cap B_{\beta} = \emptyset$ if $\alpha \neq \beta$. For each pair $\alpha, \beta \in Y$ with $\alpha \geq \beta$, let $\psi_{\alpha, \beta}: B_{\alpha} \rightarrow B_{\beta}$ be a morphism, which we call a \textit{connecting morphism}, and assume that the following conditions hold: 
\begin{enumerate}[label=(\roman*)]
\item if $\alpha \in Y$ then $\psi_{\alpha, \alpha} = 1_{B_{\alpha}}$, the identity automorphism of $B_{\alpha}$; 
\item if $\alpha, \beta, \gamma\in Y$ are such that $\alpha \geq \beta \geq \gamma$ then $\psi_{\alpha, \beta} \psi_{\beta, \gamma} = \psi_{\alpha, \gamma}$.
\end{enumerate} 

That is, the connecting morphisms compose transitively. On the set $B=\bigcup_{\alpha \in Y} B_{\alpha}$ define a multiplication by 
\[ a * b = (a \psi_{\alpha, \alpha \beta})(b \psi_{\beta, \alpha \beta})
\] 
for $a\in B_{\alpha}, b \in B_{\beta}$, and denote the resulting structure by  $B=[Y;B_{\alpha}; \psi_{\alpha, \beta}]$. Then $B$  is a band, and is called a \textit{strong semilattice $Y$ of rectangular bands $B_{\alpha}$} ($\alpha\in Y$). 

 Note also that for $e_{\alpha}\in B_{\alpha},f_{\beta}\in B_{\beta}$ we have $e_{\alpha} \geq f_{\beta} $ if and only if $\alpha \geq \beta$ and $e_{\alpha}\psi_{\alpha,\beta}=f_{\beta}$ (\cite[Lemma IV.1.7]{Petrich99}).

\begin{lemma}[{{\cite[Proposition 4.6.14]{Howie94}}}] A band is normal if and only if it is isomorphic to a strong semilattice of rectangular bands. 
\end{lemma} 

It then follows that a band is left normal if and only if it is normal with $\mathcal{D}$-classes being left zero (dually for right normal bands). The varieties of (left/right) regular bands are considered in the next section. 

\section{Homogeneous bands}\label{sec:homog basic} 

Recall that a band $B$ is homogeneous if every isomorphism between finite (equivalently, f.g.) subbands extends to an automorphism of $B$. In this section, we obtain a collection of general results on homogeneous bands. We also consider a weaker form of homogeneity by saying that a band $B$ is $n$-homogeneous if every isomorphism between subbands of cardinality $n$ extends to an automorphism of $B$, where $n\in \mathbb{N}$. Clearly, a band is homogeneous if and only if it is $n$-homogeneous for all $n\in \mathbb{N}$. 

The following isomorphism theorem for bands is folklore but proven here for completeness.

\begin{proposition} \label{iso band} Let $B=\bigcup_{\alpha\in Y} B_{\alpha}$ and $B'=\bigcup_{\alpha'\in Y'}B'_{\alpha'}$ be a pair of bands. Then for any morphism $\theta:B\rightarrow B'$, there exists a morphism $\pi:Y\rightarrow Y'$ and, for every $\alpha\in Y$, a morphism $\theta_{\alpha}:B_{\alpha}\rightarrow B'_{\alpha\pi}$ such that $\theta=\bigcup_{\alpha\in Y}\theta_{\alpha}$. 
 We denote $\theta$ as $[\theta_{\alpha},\pi]_{\alpha\in Y}$.  
\end{proposition} 

\begin{proof} The relation $\mathcal{D}$ is preserved under morphisms, that is, if $e \, \mathcal{D} \, f$ in $B$ then $e\theta \, \mathcal{D} \, f\theta$ in $B'$ for any morphisms $\theta:B\rightarrow B'$. The results then follows from Proposition \ref{structure bands}.
\end{proof} 

We shall call $\pi$ the \textit{induced (semilattice) morphism} of $\theta$. Note also that the converse of the proposition above is not true in general, that is, given any morphism $\pi:Y\rightarrow Y'$ and collection $ \theta_{\alpha}:B_{\alpha}\rightarrow B_{\alpha\pi}$ of morphisms for each $\alpha \in Y$, then $\bigcup_{\alpha\in Y}\theta_{\alpha}$ need not be a morphism. 

By considering the relationship between the automorphisms of a band and the induced automorphisms of its structure semilattice, we may define a stronger form of homogeneity for bands: structure-homogeneity. 

\begin{definition}\label{defn SH} A band $B=\bigcup_{\alpha\in Y}B_{\alpha}$ is called \textit{structure-homogeneous} if given any pair of finite subbands $A_i=\bigcup_{\alpha\in Z_i}A_{\alpha}^i$ ($i=1,2$) and any isomorphism $\theta=[\theta_{\alpha},\pi]_{\alpha\in Z_1}$ from $A_1$ to $A_2$, then for any automorphism $\hat{\pi}$ extending $\pi$, there exists an automorphism $[\hat{\theta}_{\alpha},\hat{\pi}]_{\alpha\in Y}$ of $B$ extending $\theta$. 
\end{definition} 

Note that if $B$ is structure-homogeneous then for each automorphism $\pi$ of $Y$ there exists an automorphism of $B$ with $\pi$ as induced automorphism. Indeed, fix any $\alpha\in Y$, $e_{\alpha}\in B_{\alpha}$ and $e_{\alpha\pi}\in B_{\alpha\pi}$. Then $\{e_{\alpha}\}$ and $\{e_{\alpha\pi}\}$ are isomorphic trivial bands, and so as $\pi$ extends the isomorphism from $\{\alpha\}$ to $\{\alpha\pi\}$, we may extend (by the structure-homogeneity of $B$) the isomorphism from $\{e_{\alpha}\}$ to $\{e_{\alpha\pi}\}$ to an automorphism of $B$ with induced automorphism $\pi$. 

This new concept of homogeneity will be of particular use when considering spined products of bands, which we now define. 

\begin{definition} Let $S$ and $T$ be semigroups having a common morphic image $H$, and let $\phi$ and $\psi$ be morphisms from $S$ and $T$ to $H$, respectively. The \textit{spined product} of $S$ and $T$ with respect to $H, \phi, \psi$ is defined as the semigroup 
\[ S \bowtie T:=\{(s,t)\in S \times T: s\phi = t\psi\}.
\] 
\end{definition} 

 Kimura showed in \cite{Kimura58} that a band $B$ is regular if and only if it is a spined product of a left regular and right regular band (known as the \textit{left and right component} of $B$, respectively). Moreover, a band is left (right) regular if and only if it is a semilattice of left zero (right zero) semigroups. 

Since the classes of left regular, right regular and regular bands form varieties, it can be shown that if $A$ is a subband of a regular band $L\bowtie R$, then there exist subbands $L'$ of $L$ and $R'$ of $R$ such that $A=L'\bowtie R'$.  

We shall prove in this paper that a homogeneous band is necessarily regular. The following result was first proven by Kimura, and a simplified form can be found in \cite[Lemma V.1.10]{Petrich99}: 

\begin{proposition} \label{iso regular} Let  $B=L\bowtie R$ and $B'= L' \bowtie R'$ be a pair of regular bands with structure semilattices $Y$ and $Y'$, respectively. Let $\theta^l:L\rightarrow L'$ and $\theta^r:R\rightarrow R'$ be morphisms which induce the same morphism $\pi:Y\rightarrow Y'$. Define a mapping $\theta$ by 
\[ (l,r)\theta=(l\theta^l,r\theta^r)  \quad  ((l,r)\in B).
\] 
Then $\theta$ is an morphism from $B$ onto $B'$, denoted $\theta=\theta^l \bowtie \theta^r$, and every morphism from $B$ to $B'$ can be so obtained for unique $\theta^l$ and $\theta^r$. Moreover, $\pi,\theta^l$, and $\theta^r$ are surjective/injective if and only if $\theta$ is.  
\end{proposition} 

In general, a pair of regular bands with isomorphic left and right components need not be isomorphic. The ensuing lemma gives a condition on the components of the bands which forces them to be isomorphic.  

\begin{corollary}\label{SH regular}  Let  $B=L\bowtie R$ and $B'= L' \bowtie R'$ be a pair of regular bands with structure semilattices $Y$ and $Y'$, respectively, and with $L$ and $L'$ structure-homogeneous. Then $B\cong B'$ if and only if $L \cong L'$ and $R\cong R'$ (dually for $R$ and $R'$). 
\end{corollary} 

\begin{proof} Let $\theta^r:R\rightarrow R'$ and $\theta^l:L\rightarrow L'$ be isomorphisms with induced isomorphisms $\pi_r$ and $\pi_l$ of $Y$ into $Y'$, respectively. Then as $L'$ is structure-homogeneous there exists an automorphism $\phi^l$ of $L'$ with induced automorphism $\pi_l^{-1}\pi_r$ of $Y'$. Hence $\theta^l\phi^l$ is an isomorphism from $L$ to $L'$ with induced isomorphism $\pi_l(\pi_l^{-1}\pi_r)=\pi_r$, and so $\theta^l\phi^l \bowtie \theta^r$ is an isomorphism from $B$ to $B'$ by Proposition \ref{iso regular}. The converse is immediate from Proposition \ref{iso regular}. 
\end{proof}
 
\begin{corollary} \label{structure homog reg} Let $B$ be the spined product of a homogeneous left regular band $L$ and homogeneous right regular band $R$. If either $L$ or $R$ are structure-homogeneous, then $B$ is homogeneous. Moreover, if both $L$ and $R$ are structure-homogeneous, then so is $B$. 
\end{corollary}

\begin{proof} Suppose w.l.o.g. that $L$ is structure-homogeneous, with structure semilattice $Y$. Let $\theta=\theta^l\bowtie \theta^r$ be an isomorphism between finite subbands $A_1=L_1\bowtie R_1$ and $A_2=L_2\bowtie R_2$ of $B$. Then by Proposition \ref{iso regular}, the isomorphisms $\theta^l$ and $ \theta^r$ both induce an isomorphism $\pi$ between the structure semilattices $Y_1$ and $Y_2$ of $A_1$ and $A_2$, respectively. Since $R$ is homogeneous, we may extend $\theta^r:R_1\rightarrow R_2$ to an automorphism $\bar{\theta}^r$ of $R$, with induced automorphism $\bar{\pi}$ of $Y$ extending $\pi$. Since $L$ is structure-homogeneous, there exists an automorphism $\bar{\theta}^l$ of $L$ extending $\theta^l$ and with induced automorphism $\bar{\pi}$ of $Y$. Hence $\bar{\theta}^l\bowtie \bar{\theta}^r$ is an automorphism of $B$, which extends $\theta$ as required. 
The final result is proven in a similar fashion. 
\end{proof} 

In later sections, this will give us a useful method for building homogeneous regular bands from left regular and right regular bands.
 However, at this stage, it is unclear how the homogeneity of a regular band is affected by the homogeneity of its left and right components. 

 Given a band $B=\bigcup_{\alpha\in Y} B_{\alpha}$, we now define a number of crucial subsets of $B$. If $\alpha>\beta$ in $Y$ and $e_{\alpha}\in B_{\alpha}$ then we denote
\begin{enumerate}[label=(\roman*)]
\item $  B_{\beta}(e_{\alpha}): =\{e_{\beta}\in B_{\beta}:e_{\beta}<e_{\alpha}\};$ 
 \item $B_{\alpha,\beta}: =\bigcup_{f_\alpha\in B_{\alpha}} B_{\beta}(f_{\alpha})=\{e_{\beta}\in B_{\beta}:e_{\beta}<f_{\alpha} \text{ for some } f_{\alpha}\in B_{\alpha}\};$
 \item $\mathcal{R}(B_{\beta}(e_{\alpha})): =\{f_{\beta}\in B_{\beta}:\exists e_{\beta}\in B_{\beta}(e_{\alpha}), f_{\beta} \, \mathcal{R} \, e_{\beta}\}=\{f_{\beta}\in B_{\beta}:f_{\beta}<_r e_{\alpha}\};$ 
\end{enumerate}
 
and dually for $\mathcal{L}(B_{\beta}(e_{\alpha}))$.  

Note first that 
\[ \mathcal{R}(B_{\beta}(e_{\alpha}))=\{f_{\beta}\in B_{\beta}:\exists e_{\beta}\in B_{\beta}(e_{\alpha}), f_{\beta} \, \mathcal{R} \, e_{\beta}\}.
\] 
Indeed, we claim that if $f_{\beta}\in B_\beta$ and $e_\alpha\in B_\alpha$, then  $f_{\beta}<_r e_{\alpha}$ if and only if $f_{\beta} \, \mathcal{R} \, e_{\alpha}f_{\beta}e_{\alpha}$, to which the result follows as $e_{\alpha}f_{\beta}e_{\alpha}<e_{\alpha}$.
 If $f_{\beta}<_r e_{\alpha}$ then $f_{\beta}=e_{\alpha}f_{\beta}$, and so 
\[ f_{\beta}(e_{\alpha}f_{\beta}e_{\alpha})= (f_{\beta}e_{\alpha})(f_{\beta}e_{\alpha})= f_{\beta}e_{\alpha} = e_{\alpha}f_{\beta} e_{\alpha} 
\] 
so that $f_{\beta} >_{r} e_{\alpha}f_{\beta}e_{\alpha}$. Hence $f_{\beta}$ and $e_{\alpha}f_{\beta}e_{\alpha}$, being elements of $B_{\beta}$, are $\mathcal{R}$-related.  Conversely, if $f_{\beta} \, \mathcal{R} \, e_{\alpha}f_{\beta}e_{\alpha}$ then 
$f_{\beta} = (e_{\alpha}f_{\beta} e_{\alpha})f_\beta = e_{\alpha}f_{\beta}$ and the claim holds. 
We observe that each $B_{\beta}(e_{\alpha})$ is non-empty, since for any $e_{\beta}\in B_{\beta}$ we have $e_{\alpha}>e_{\alpha}e_{\beta}e_{\alpha}\in B_{\beta}$. Moreover, each of the sets defined above are subbands of $B_{\beta}$. Indeed, if $e_{\alpha}>e_{\beta}$ and $f_{\alpha}>f_{\beta}$ then 
\[ e_{\beta}f_{\beta}e_{\alpha}f_{\alpha}=e_{\beta}(f_{\beta}f_{\alpha})e_{\alpha}f_{\alpha}= e_{\beta}f_{\beta}(f_{\alpha}e_{\alpha}f_{\alpha})=e_{\beta}f_{\beta}f_{\alpha}=e_{\beta}f_{\beta}, 
\] 
and similarly $e_{\alpha}f_{\alpha}e_{\beta}f_{\beta}=e_{\beta}f_{\beta}$. Hence $e_{\alpha}f_{\alpha}>e_{\beta}f_{\beta}\in B_{\alpha,\beta}$, and so $B_{\alpha,\beta}$ is a subband. By taking $e_{\alpha}=f_{\alpha}$ gives $B_{\beta}(e_{\alpha})$ to be a subband. Finally $ \mathcal{R}(B_{\beta}(e_{\alpha}))$, being a collection of $ \mathcal{R}$-classes of $B_{\beta}$, is a subband. 

\begin{definition} Given a subset $A$ of a band $B$, we define the \textit{support of} $A$ as 
\[ \text{supp}(A):=\{\gamma\in Y: A\cap B_{\gamma}\neq \emptyset\}. 
\] 
\end{definition} 
If $A$ is a subband of $B$ then clearly supp($A$) is simply the structure semilattice of $A$. 

Given a semilattice $Y$ and $\beta\in Y$, we call $\alpha\in Y$ a \textit{cover} of $\beta$ if $\alpha>\beta$ and whenever $\alpha\geq \delta \geq \beta$ for some $\delta\in Y$, then $\alpha=\delta$ or $\delta=\beta$. A semilattice in which no element has a cover is called \textit{dense}. Note that a countable dense linear order without endpoints is isomorphic to $(\mathbb{Q},\leq)$. 

\begin{corollary}\label{basics} Let $B=\bigcup_{\alpha\in Y} B_{\alpha}$ be a homogeneous band where $Y$ is non-trivial. Then, for all $\alpha>\beta$ and $\alpha'> \beta'$ in $Y$, $e_{\alpha}\in B_{\alpha}$ and $e_{\alpha'}\in B_{\alpha'}$, we have 
\begin{enumerate}[label=(\roman*), font=\normalfont]
\item For every $e,f\in B$ there exists an automorphism of $B$ mapping $e$ to $f$. 
\item $ Y$  is dense and without maximal or minimal elements; 
\item  $B_{\alpha}\cong B_{\alpha'}, L_{\alpha}\cong L_{\alpha'}$ and $R_{\alpha}\cong R_{\alpha'};$ 
\item $B_{\beta}(e_{\alpha})\cong B_{\beta'}(e_{\alpha'})$.
\end{enumerate}
\end{corollary} 

\begin{proof} $(\mathrm{i})$ For any $e,f\in B$ we have $\{e\}\cong \{f\}$, so the result follows as $B$ is 1-homogeneous. 

$(\mathrm{ii})$  Suppose for contradiction that $\alpha$ is maximal, and let $\beta<\alpha$ in $Y$.
 Then for any $e_{\alpha}\in B_{\alpha}$ and $e_{\beta}\in B_{\beta}$, there exists by $(\mathrm{i})$ an automorphism mapping $e_{\alpha}$ to $e_{\beta}$. Hence by Proposition \ref{iso band} there exists an automorphism of $Y$ mapping $\alpha$ to $\beta$, contradicting $\alpha$ being maximal. The result is proven similarly for minimal elements. 

Now suppose $\alpha>\beta$ in $Y$. Since $\beta$ is not minimum, there exists $\gamma\in Y$ such that $\gamma<\beta$. Let $e_{\alpha}\in B_{\alpha}, e_{\beta}\in B_{\beta}(e_{\alpha})$ and $e_{\gamma}\in B_{\gamma}(e_{\beta})$, so $e_{\gamma}\in B_{\gamma}(e_{\alpha})$. Then by extending the isomorphism from  $\{e_{\alpha},e_{\gamma}\}$ to $\{e_{\alpha},e_{\beta}\}$ to an automorphism of $B$, it follows by considering the image of $\beta$ under the induced automorphism of $Y$ that there exists $\gamma'\in Y$ such that $ \alpha>\gamma'>\beta$, and so  $\alpha$ does not cover $\beta$. Hence $Y$ is dense. 

$(\mathrm{iii})$  By taking an automorphism $\theta$ of $B$ which sends $e_{\alpha}$ to $e_{\alpha'}$, it follows from Proposition \ref{iso band} that $B_{\alpha}\theta=B_{\alpha'}$.  The results for $L_{\alpha}$ and $R_{\alpha}$ are then immediate from Proposition \ref{rec bands iso}.

$(\mathrm{iv})$ Let $e_{\beta}\in B_{\beta}(e_{\alpha})$ and $e_{\beta'}\in B_{\beta'}(e_{\alpha'})$. Since $\{e_{\alpha},e_{\beta}\}$ and $\{e_{\alpha'},e_{\beta'}\}$ are isomorphic subbands, the result follows by extending the unique isomorphism between them to an automorphism of $B$. 
\end{proof}

If $B$ is a band with non-trivial semilattice, then by the proof of $(\mathrm{ii})$, it is clear that $B$, regarded as a poset under its natural order, cannot have maximal or minimal elements, since then natural order is preserved under automorphisms of $B$. 

One of our fundamental questions in this article is whether or not the homogeneity of a band is inherited by its structure semilattice.
 The answer is yes, but surprisingly we have not been able to find a direct proof. For now we are only able to partially answer this question: 

\begin{proposition}\label{linear homog} If $B=\bigcup_{\alpha\in Y} B_{\alpha}$ is a homogeneous band, then $Y$ is 2-homogeneous. Consequently, if $Y$ is linearly or semi-linearly ordered then $Y$ is homogeneous.
\end{proposition} 

\begin{proof}
As the unique (up to isomorphism) 2 element semilattice is a chain, it suffices to consider a pair $\alpha_i>\beta_i$ ($i=1,2$) in $Y$. Fix $e_{\alpha_i}\in B_{\alpha_i}$ for each $i=1,2$ and consider any $e_{\beta_i}\in B_{\beta_i}(e_{\alpha_i})$. By extending the isomorphism between $\{e_{\alpha_1},e_{\beta_1}\}$ and $\{e_{\alpha_2},e_{\beta_2}\}$, it follows by Proposition \ref{iso band} that $Y$ is 2-homogeneous. The final result is then immediate from \cite[Proposition 2.1]{Truss99}. 
\end{proof}

To avoid falling into already complete classifications, we assume throughout this paper that $Y$ is non-trivial (so $B$ is not a rectangular band) and each $\mathcal{D}$-class is non-trivial (so $B$ is not a semilattice). 

\section{Homogeneous normal bands} \label{sec:normal}

 In this section we classify homogeneous normal bands. Our aim is helped by not only a classification theorem for normal bands which gives the local structure, but also a relatively simple isomorphism theorem (see \cite[Lemma IV.1.8]{Petrich99}, for example): 

\begin{proposition}\label{iso normal} Let $B=[Y;B_{\alpha};\psi_{\alpha,\beta}]$ and $B'=[Y';B_{\alpha'}';\psi_{\alpha',\beta'}']$ be a pair of normal bands. Let $\pi:Y\rightarrow Y'$ be an isomorphism, and for every $\alpha\in Y$, let $\theta_{\alpha}:B_{\alpha}\rightarrow B_{\alpha\pi}$ be an isomorphism such that for any $\alpha\geq \beta$ in $Y$, the diagram
\begin{align}\label{diagram}  \xymatrix{
B_{\alpha} \ar[d]^{\psi_{\alpha, \beta}} \ar[r]^{\theta_{\alpha}} &B_{\alpha \pi}' \ar[d]^{\psi_{\alpha \pi, \beta \pi}'} \\\
B_{\beta} \ar[r]^{\theta_{\beta}} &B_{\beta \pi}'}
\end{align}  
commutes. Then $\theta=[\theta_{\alpha},\pi]_{\alpha\in Y}$ is an isomorphism from $B$ into $B'$. Conversely, every isomorphism of $B$ into $B'$ can be so obtained for unique $\pi$ and $\theta_{\alpha}$. 
\end{proposition} 

We denote the diagram \eqref{diagram} as $[\alpha,\beta;\alpha\pi,\beta\pi]$.

To understand the homogeneity of normal bands, we require a better understanding of the finite subbands. Since the class of all normal bands forms a variety, the following lemma is easily verified.

\begin{lemma} \label{sub normal} Let $A$ be a subband of a normal band $B=[Y;B_{\alpha};\psi_{\alpha,\beta}]$. Then $A=[Z;A_{\alpha};\psi_{\alpha,\beta}|_A]$, for some subsemilattice $Z$ of $Y$ and rectangular subbands $A_{\alpha}$ of $B_{\alpha}$ $(\alpha\in Z)$ for each $\alpha\in Z$.
\end{lemma} 

Given a normal band $B=[Y;B_{\alpha};\psi_{\alpha,\beta}]$, we shall denote Im $\psi_{\alpha,\beta}$ as $I_{\alpha,\beta}$, or $I_{\alpha,\beta}^B$ if we need to distinguish the band $B$. 
 
\begin{lemma} \label{D full} Let $B=[Y;B_{\alpha};\psi_{\alpha,\beta}]$ be a homogeneous normal band. Then $B_{\beta} = \bigcup_{\alpha>\beta} I_{\alpha,\beta}$ for each $\beta\in Y$. 
\end{lemma} 

\begin{proof} As a consequence of Corollary \ref{basics}, $B$ contains no maximal elements under its natural order. The result then follows as $e_{\alpha}\geq e_{\beta}$ if and only if $\alpha\geq \beta$ and $e_{\alpha}\psi_{\alpha,\beta}=e_{\beta}$. 
 \end{proof} 
 
\begin{lemma}\label{image same} Let $B=[Y;B_{\alpha};\psi_{\alpha,\beta}]$ be a homogeneous normal band. If $\alpha_i>\beta_i$ $(i=1,2)$ in $Y$ then there exists isomorphisms $\theta_{\alpha_1}:B_{\alpha_1}\rightarrow B_{\alpha_2}$ and $\theta_{\beta_1}:B_{\beta_1} \rightarrow B_{\beta_2}$ such that 
\[ \theta_{\alpha_1}\psi_{\alpha_2,\beta_2} = \psi_{\alpha_1,\beta_1}\theta_{\beta_1}
\]
In particular $I_{\alpha_1,\beta_1}\cong I_{\alpha_2,\beta_2}$ and $\psi_{\alpha_1,\beta_1}$ is surjective/injective if and only if $\psi_{\alpha_2,\beta_2}$ is also. 
\end{lemma} 

\begin{proof}
Let $e_{\alpha_i}\in B_{\alpha_i}$ ($i=1,2$) be fixed. By extending the unique isomorphism between the 2 element subbands $\{ e_{\alpha_1},e_{\alpha_1}\psi_{\alpha_1,\beta_1}\}$ and  $\{ e_{\alpha_2},e_{\alpha_2}\psi_{\alpha_2,\beta_2}\}$ to an automorphism of $B$, the diagram $[\alpha_1,\beta_1;\alpha_2,\beta_2]$ commutes by Proposition \ref{iso normal}, and the result is immediate. 
\end{proof}

Since a normal band is regular, it can be regarded as the spined product of a left regular and right regular band. The following results and subsequent Proposition \ref{normal spined} are obtained from \cite[Proposition 4.6.17]{Howie94}. Let $B=[Y;B_{\alpha};\psi_{\alpha,\beta}]$ be a normal band, where $B_{\alpha}=L_{\alpha}\times R_{\alpha}$ for some left zero semigroup $L_{\alpha}$ and right zero semigroup $R_{\alpha}$. Then the connecting morphisms determine morphisms  $ \psi_{\alpha,\beta}^l:L_{\alpha}\rightarrow L_{\beta}$ and $\psi_{\alpha,\beta}^r:R_{\alpha}\rightarrow R_{\beta}$ such that
\begin{equation}\label{spined a} (l_{\alpha},r_{\alpha})\psi_{\alpha,\beta}=(l_{\alpha}\psi_{\alpha,\beta}^l,r_{\alpha}\psi_{\alpha,\beta}^r)
\end{equation}
for every $(l_{\alpha},r_{\alpha})\in B_{\alpha}$. Moreover, $L=\bigcup \{L_{\alpha}:\alpha\in Y\}$ becomes a strong semilattice of left zero semigroups $[Y;L_{\alpha};\psi_{\alpha,\beta}^l]$ under $\circ$, where for $l_{\alpha}\in L_{\alpha}, l_{\beta}\in L_{\beta}$, 
\[ l_{\alpha} \circ l_{\beta}=(l_{\alpha}\psi_{\alpha,\alpha\beta}^l)(l_{\beta}\psi_{\beta,\alpha\beta}^l)=l_{\alpha}\psi_{\alpha,\alpha\beta}^l
\] 
since $L_{\alpha\beta}$ is left zero (dually for $R$). Take morphisms $\phi:L\rightarrow Y, \psi:R\rightarrow Y$ given by $l_{\alpha}\phi=\alpha$ and $r_{\alpha}\psi=\alpha$. Then the spined product of $L$ and $R$ consisting of pairs $(l,r)$ for which $l\phi=r\psi$ coincides with 
\[ \bigcup \{L_{\alpha}\times R_{\alpha}:\alpha\in Y\}.  
\] 
It then follows from \eqref{spined a} that $B=L \bowtie R$ and so: 

\begin{proposition}\label{normal spined} Every normal band $B$ is isomorphic to a spined product of a left normal and a right normal band. 
\end{proposition} 

Consequently, by Proposition \ref{iso regular}, a pair of normal bands $L\bowtie R$ and $L'\bowtie R'$ are isomorphic if and only if there exists an isomorphism from $L$ to $L'$ and $R$ to $R'$ with the same induced isomorphism between the structure semilattices.

A normal band is called an \textit{image-trivial normal band} if the images of the non-identity connecting morphisms all have cardinality 1. A normal band is called a \textit{surjective normal band} if each connecting morphism is surjective. Note that a normal band is both image-trivial and surjective if and only if it is a semilattice.  Moreover, a normal band $L\bowtie R$ is an image-trivial/surjective normal band if and only if both $L$ and $R$ are also image-trivial/surjective.

\begin{lemma}\label{trivial or surj} Let $B=[Y;L_{\alpha};\psi_{\alpha,\beta}^l]\bowtie [Y;R_{\alpha};\psi_{\alpha,\beta}^r]=L\bowtie R$ be a homogeneous normal band. Then $R$ is either an image-trivial or surjective right normal band (dually for $L$).
\end{lemma} 

\begin{proof} If $R$ is a semilattice then $B$ is isomorphic to $L$, and so the result is immediate. Assume instead that $|R_{\alpha}|>1$ for some $\alpha\in Y$. Then $|R_{\alpha}|>1$ for all $\alpha\in Y$ by Corollary \ref{basics} ($\mathrm{iii}$). Suppose there exists $\alpha>\beta$ in $Y$ such that $I^R_{\alpha,\beta} \neq R_{\beta}$.
 Let $r_{\alpha}\psi^r_{\alpha,\beta}=r_{\beta}$, $s_{\alpha}\psi^r_{\alpha,\beta}=s_{\beta}$ (with $r_{\alpha}\neq s_{\alpha}$) and $t_{\beta}\not\in I^R_{\alpha,\beta}$. Fix $l_{\alpha}\in L_{\alpha}$ and let $l_{\alpha}\psi_{\alpha,\beta}^l=l_{\beta}$. Note that for any $x_{\beta}\in R_{\beta}$ we have 
\[ (l_{\alpha},r_{\alpha})(l_{\beta},x_{\beta})=(l_{\beta},r_{\beta}x_{\beta})=(l_{\beta},x_{\beta}) \quad \text{and} \quad (l_{\beta},x_{\beta})(l_{\alpha},r_{\alpha})=(l_{\beta},r_{\beta}).
\] 
Hence if $x_{\beta}\neq r_{\beta}$ then $\langle (l_{\alpha},r_{\alpha}),(l_{\beta},x_{\beta})\rangle =\{(l_{\alpha},r_{\alpha}),(l_{\beta},x_{\beta}),(l_{\beta},r_{\beta}) \}$ is a 3 element subband. In particular, if $s_{\beta}\neq r_{\beta}$ then the map
\[ \phi:\langle (l_{\alpha},r_{\alpha}),(l_{\beta},s_{\beta})\rangle \rightarrow \langle (l_{\alpha},r_{\alpha}), (l_{\beta}, t_{\beta}) \rangle
\] 
fixing $(l_{\alpha},r_{\alpha})$  and such that $(l_{\beta},s_{\beta})\phi=(l_{\beta},t_{\beta})$ is clearly an isomorphism. Extend $\phi$ to an automorphism $\theta^l \bowtie \theta^r$ of $B$. Then as $\theta^r=[\theta_{\alpha}^r,\pi]_{\alpha\in Y}$ is an automorphism of $R$ we have, by the commutativity of $[\alpha,\beta;\alpha,\beta]$ in $R$, 
\[ (s_{\alpha}\theta^r_{\alpha})\psi^r_{\alpha,\beta}=(s_{\alpha}\psi^r_{\alpha,\beta})\theta^r_{\beta}=s_{\beta}\theta_{\beta}^r=t_{\beta},
\] 
contradicting $t_{\beta}\not\in I^R_{\alpha,\beta}$. Thus $s_{\beta}=r_{\beta}$, so that $I_{\alpha,\beta}$ has cardinality 1,   and so $R$ is an image-trivial normal band by Lemma \ref{image same}. The dual gives the result for left normal bands. 
\end{proof}

Hence if $B=L\bowtie R$ is a homogeneous normal band then $B$ is either an image-trivial normal band (if $L$ and $R$ are image-trivial), or the images of the connecting morphisms are a single $\mathcal{L/R}$-class (if $L$/$R$ is a surjective normal band and $R/L$ is an image-trivial normal band) or $\mathcal{D}$-class (if $L$ and $R$ are surjective normal bands).

We split our classification of homogeneous normal bands into three parts. In Section \ref{sec:trivial} we classify homogeneous image-trivial normal bands, and in Section \ref{sec:surj} homogeneous surjective normal bands. Using the results obtained in these sections, the final case (and its dual) is easily obtained at the end of Section \ref{sec:surj}. 

\subsection{Image-trivial normal bands} \label{sec:trivial}

In this section we are concerned with the classification of image-trivial homogeneous normal bands. 

Where convenient, we denote an image-trivial normal band $[Y;B_{\alpha};\psi_{\alpha,\beta}]$ such that $I_{\alpha,\beta}=\{\epsilon_{\alpha,\beta}\}$ for each $\alpha>\beta$, as $[Y;B_{\alpha};\psi_{\alpha,\beta};\epsilon_{\alpha,\beta}]$. Note that if $\alpha,\gamma>\beta$ in $Y$ are such that $\alpha\gamma>\beta$ then $B_{\alpha}\psi_{\alpha,\beta}=(B_{\alpha}\psi_{\alpha,\alpha\gamma})\psi_{\alpha\gamma,\beta}=\{\epsilon_{\alpha\gamma,\beta}\} = (B_{\gamma}\psi_{\gamma,\alpha\gamma})\psi_{\alpha\gamma,\beta}=B_{\gamma}\psi_{\gamma,\beta}$
and so 
\begin{align}\label{eq trivial} 
\epsilon_{\alpha,\beta}=\epsilon_{\alpha\gamma,\beta}=\epsilon_{\gamma,\beta}. 
\end{align}
Notice that \eqref{eq trivial} automatically holds if $\alpha>\gamma>\beta$. 

If $Y=\mathbb{Q}$ then for any $\beta\in \mathbb{Q}$ and $\alpha,\gamma>\beta$ we have $\epsilon_{\alpha,\beta}=\epsilon_{\gamma,\beta}$ by \eqref{eq trivial}. Hence any $e_{\beta}\in B_{\beta}\setminus \{\epsilon_{\alpha,\beta}\}$ is a maximal element in the poset $(B,\leq)$. Consequently, an  image-trivial homogeneous normal band with a linear structure semilattice is isomorphic to $\mathbb{Q}$ by Corollary \ref{basics}.   

While the following lemma is stronger than what is required in this section, it will be vital for later results, and the generalization adds little extra work. 

\begin{lemma} \label{trivial not universal} Let $B=[Y;B_{\alpha};\psi_{\alpha,\beta}]=L\bowtie R$ be a homogeneous normal band such that either $L$ or $R$ is a non-semilattice image-trivial normal band. Then $Y$ is a homogeneous semilinear order. 
\end{lemma} 

\begin{proof} Assume w.l.o.g. that $L=[Y;L_{\alpha};\psi_{\alpha,\beta}^l;\epsilon_{\alpha,\beta}^l]$ is a non-semilattice image-trivial normal band, so that $|L_{\alpha}|>1$ for all $\alpha\in Y$ by Corollary \ref{basics} ($\mathrm{iii}$).
 Note that $R=[Y;R_{\alpha};\psi_{\alpha,\beta}^r]$ is image-trivial or surjective by Lemma \ref{trivial or surj}.
  Suppose for contradiction that $Y$ contains a diamond $D=\{\delta,\alpha,\gamma,\beta\}$, where $\delta>\{\alpha,\gamma\}>\beta$. Fix $e_{\delta}=(l_{\delta},r_{\delta})\in B_{\delta}$ and let 
\[ e_{\alpha}=e_{\delta}\psi_{\delta,\alpha}=(\epsilon_{\delta,\alpha}^l,r_{\delta}\psi_{\delta,\alpha}^r), \quad  e_{\gamma}=e_{\delta}\psi_{\delta,\gamma}=(\epsilon_{\delta,\gamma}^l,r_{\delta}\psi_{\delta,\gamma}^r), \quad e_{\beta}=e_{\delta}\psi_{\delta,\beta}= (\epsilon_{\delta,\beta}^l,r_{\delta}\psi_{\delta,\beta}^r),
\] 
noting that $\epsilon_{\alpha,\beta}^l=\epsilon_{\delta,\beta}^l= \epsilon_{\gamma,\beta}^l$ by \eqref{eq trivial}. By construction $\{e_{\delta},e_{\alpha},e_{\gamma},e_{\beta}\}$ is  isomorphic to $D$. If there exists $l_{\beta}\in L_{\beta}\setminus \{\epsilon_{\delta,\beta}^l\}$, then by Lemma \ref{D full} there exists $\tau>\beta$ such that $l_{\beta}=\epsilon_{\tau,\beta}^l$. Note that $\alpha\tau=\beta$, since if $\alpha\tau>\beta$ then $l_{\beta}=\epsilon_{\tau,\beta}^l=\epsilon_{\alpha,\beta}^l$ by \eqref{eq trivial}. Let $\kappa<\beta$ and $e_{\kappa}=e_{\beta}\psi_{\beta,\kappa}=(\epsilon_{\beta,\kappa}^l,r_{\delta}\psi_{\delta,\kappa}^r)$. Extend the unique isomorphism between the 3-chains $e_{\delta}>e_{\alpha}>e_{\beta}$ and $e_{\alpha}>e_{\beta}>e_{\kappa}$ to an automorphism $\theta=[\theta_{\alpha},\pi]_{\alpha\in Y}$ of $B$. Let $e_{\rho}=e_{\gamma}\theta>e_{\beta}\theta=e_{\kappa}$ for some $\rho\in Y$. Then $\alpha>\rho>\kappa$ (since $\delta>\gamma>\beta$),  $\rho\beta=\kappa$ (since $\gamma\alpha=\beta$) and 
\[\rho\tau = (\rho\alpha)\tau= \rho(\alpha\tau)=\rho\beta=\kappa.
\] 
 We claim that there exists $e_{\tau}\in B_{\tau}$ such that $e_{\tau}>e_{\kappa}$. Indeed, if $R$ is also image-trivial and $B=[Y;B_{\alpha};\psi_{\alpha,\beta};\epsilon_{\alpha,\beta}]$, then the claim holds for any $e_{\tau}$ by \eqref{eq trivial}, as $\tau>\beta>\kappa$, so that $\epsilon_{\tau,\kappa}=\epsilon_{\beta,\kappa}=e_{\kappa}$. If $R$ is surjective, then there exists $r_{\tau}\in R_{\tau}$ such that $r_{\tau}\psi_{\tau,\kappa}^r=r_{\delta}\psi_{\delta,\kappa}^r$. Thus, for any $l_{\tau}$, we have $(l_{\tau},r_{\tau})\psi_{\tau,\kappa}=(\epsilon^l_{\tau,\kappa},r_{\tau}\psi^r_{\tau,\kappa})=(\epsilon^l_{\beta,\kappa},r_{\delta}\psi^r_{\delta,\kappa})=e_{\kappa}$, and so the claim is proven. Fix some $e_{\tau}>e_{\kappa}$. By extending any isomorphism between the 3 element non-chain semilattices $\langle e_{\rho},e_{\tau},e_{\kappa}\rangle$ and  $\langle e_{\alpha},e_{\gamma},e_{\beta} \rangle$, it follows that there exists $\sigma>\rho,\tau$ (as $\delta>\alpha,\gamma$). 

\begin{figure}[h]
\centering
\def\svgwidth{100pt} 
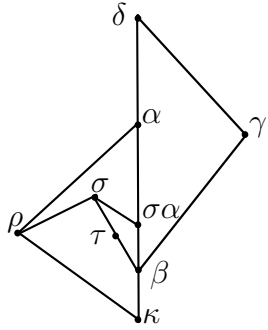 
\caption[Caption for LOF]{A subsemilattice of $Y$}
\label{subsemilattice of} 
\end{figure}

Since $\sigma>\tau>\beta$ and $\alpha>\beta$ we have $\sigma\alpha\geq \beta$. If $\sigma\alpha=\beta$ then $\beta\geq \rho$ (as $\sigma,\alpha>\rho$), and so as $\rho\beta=\kappa$ we have $\rho=\kappa$, a contradiction. Hence $\sigma\alpha>\beta$ and we thus arrive at the subsemilattice of $Y$ in Figure \ref{subsemilattice of}. Moreover,
\[ \epsilon_{\alpha,\beta}^l= \epsilon_{\sigma\alpha,\beta}^l = \epsilon_{\sigma,\beta}^l=\epsilon_{\tau,\beta}^l=l_{\beta}
\] 
by \eqref{eq trivial}, contradicting $l_{\beta} \neq\epsilon_{\delta,\beta}^l=\epsilon_{\alpha,\beta}^l$. Hence no such $l_{\beta}$ exists, and $|L_{\beta}|=1$, a contradiction. Hence, by Proposition \ref{linear homog}, $Y$ is a homogeneous linear or semilinear order. Since $B$ is a non-semilattice, the result follows by the note above the lemma. 
\end{proof}
In particular, image-trivial homogeneous normal bands have semilinear structure semilattices. To understand homogeneous image-trivial bands, it will be crucial to understand the structure of homogeneous semilinear orders. 

Let $Y$ be a dense semilinear order. We call a set $Z\subseteq Y$ \textit{connected} if for any $x,y\in Z$ there exists $z_1,\dots,z_n\in Z$  $(n\in \mathbb{N})$ with $z_1=x$, $z_n=y$, and $z_i\leq z_{i+1}$ or $z_{i+1}\leq z_i$ for all $1\leq i \leq n$. Given $\alpha\in Y$, we call the maximal connected subsets of $\{\gamma\in Y:\gamma>\alpha\}$  the \textit{cones of $\alpha$}, and let $C(\alpha)$ denote the set of all cones of $\alpha$.
 \begin{remark}[{{\cite[Remark 5.11]{Droste85}}}] \label{remark} Let $\alpha\in Y$,  $A\in C(\alpha)$ and $\gamma\in A$. Then for any $\delta\in Y$ we have $\delta \in A$ if and only if $\alpha<\delta\gamma$. 
 \end{remark} 
 Consequently, the cones of $\alpha\in Y$ partition the set $\{\gamma\in Y:\gamma>\alpha\}$. If there exists $r\in \mathbb{N}^*$ such that $|C(\alpha)|=r$ for all $\alpha\in Y$, then $r$ is known as the \textit{ramification order} of $Y$. Each homogeneous semilinear order has a ramification order \cite{Droste85}.

Let $B=[Y;B_{\alpha};\psi_{\alpha,\beta};\epsilon_{\alpha,\beta}]$ be an image-trivial normal band, where $Y$ is a semilinear order. Since $B$ is image-trivial, we may define a cone of $e_{\alpha}\in B_{\alpha}$ as a maximal connected subset of $\{\gamma\in Y:B_{\gamma}\psi_{\gamma,\alpha}=e_{\alpha}\}$. Let $C(e_{\alpha})$ denote the set of all cones of $e_{\alpha}$.  Let $\gamma,\gamma'>\alpha$ and suppose $\gamma$ is connected to $\gamma'$.  From Remark \ref{remark} we have that $\gamma\gamma'>\alpha$ and so by \eqref{eq trivial} $B_{\gamma}\psi_{\gamma,\alpha} = B_{\gamma'}\psi_{\gamma',\alpha}$. Consequently, the set $\{\gamma\in Y:\gamma>\alpha, B_{\gamma}\psi_{\gamma,\alpha}=e_{\alpha}\}$ is a union of cones of $\alpha$, and $C(\alpha)=\bigcup_{e_{\alpha}\in B_{\alpha}}C(e_{\alpha})$.

If there exists $k\in \mathbb{N}^*$ such that $|C(e_{\alpha})|=k$ for all $e_{\alpha}\in B$, then $k$ is called the \textit{ramification order} of $B$. If $B$ is homogeneous then (as $Y$ is homogeneous and $B$ is 1-homogeneous) the ramification orders exist for $Y$ and $B$, say, $r$ and $k$ respectively, and they are related according to $r=k \cdot |B_{\alpha}|$. Moreover, by Lemma \ref{D full}, $B_{\beta}=\bigcup_{\alpha>\beta}\epsilon_{\alpha,\beta}$ for each $\beta\in Y$. 

As shown in \cite[Theorem 6.21]{Droste85}, there exists for each $r\in \mathbb{N}^*$ a unique (up to isomorphism) countable homogeneous semilinear order of ramification order $r$, denoted $T_r$. Moreover, a semilinear order is isomorphic to $T_r$ if and only if it is dense and has ramification order $r$. 

We can reconstruct $T_r$ from any $\alpha\in T_r$ inductively by following the proof of Theorem 6.16 in \cite{Droste85}, as we now explain, but omitting the proof. Consider an enumeration of $T_r$ given by $T_r=\{a_i:i\in \mathbb{N}\}$, where $a_1=\alpha$. Let $Y_0=\emptyset$ and $Y_1=Z_0$ be a maximal chain in $T_r$ which contains $a_1$.
Suppose for some $i\in \mathbb{N}$, the semilattices $Y_j$ and posets $Z_{j-1}$  ($j\leq i$) have already been defined such that the following conditions hold for each $1\leq j \leq i$:
\begin{enumerate}[label=(\roman*), font=\normalfont]
\item $Y_j=Y_{j-1}\sqcup Z_{j-1} \text{ and }a_j\in Y_j$ (where $\sqcup$ denotes the disjoint union);
\item if $z\in Z_{j-1}$, then there exists a unique maximal chain $C$ in $T_r$ with
 $z \in C \subseteq Y_j$ \\ and $\{c \in C: z \leq c\} \subseteq Z_{j-1}$;
 \item if $z \in Z_{j-2} \,(j\geq 2)$ and $D$  is any cone of $z$ disjoint to $Y_{j-1}$, then $D \cap Z_{j-1}\neq \emptyset$. 
\end{enumerate}
It follows from $(\mathrm{ii})$ that whenever $1 \leq j \leq  i$, $z \in Y_j$, $y\in T_r$ and $y < z$, then $y \in Y_j$. Moreover, the conditions above trivially hold for the case $i=1$. 

If $a_{i+1}\not\in Y_i$ then it is shown that there exists $z\in Z_{i-1}$ such that $a_{i+1}$ belongs to some (unique) cone $A_z$ of $z$ which is disjoint to $Y_i$. 
For each $\beta\in Z_{i-1}$ take a maximal subchain of each cone $A\in C(\beta)$ such that $Y_i\cap A=\emptyset$, where if $\beta=z$ then we take a maximal subchain of $A_z$ which contains $a_{i+1}$. By condition $(\mathrm{ii})$ the set $\{y\in Y_{i}:\beta<y\}$ is a chain, and thus contained in a single cone of $\beta$, and so only one cone will intersect $Y_i$ non-trivially.

\begin{figure}[h]
\centering
\def\svgwidth{150pt} 
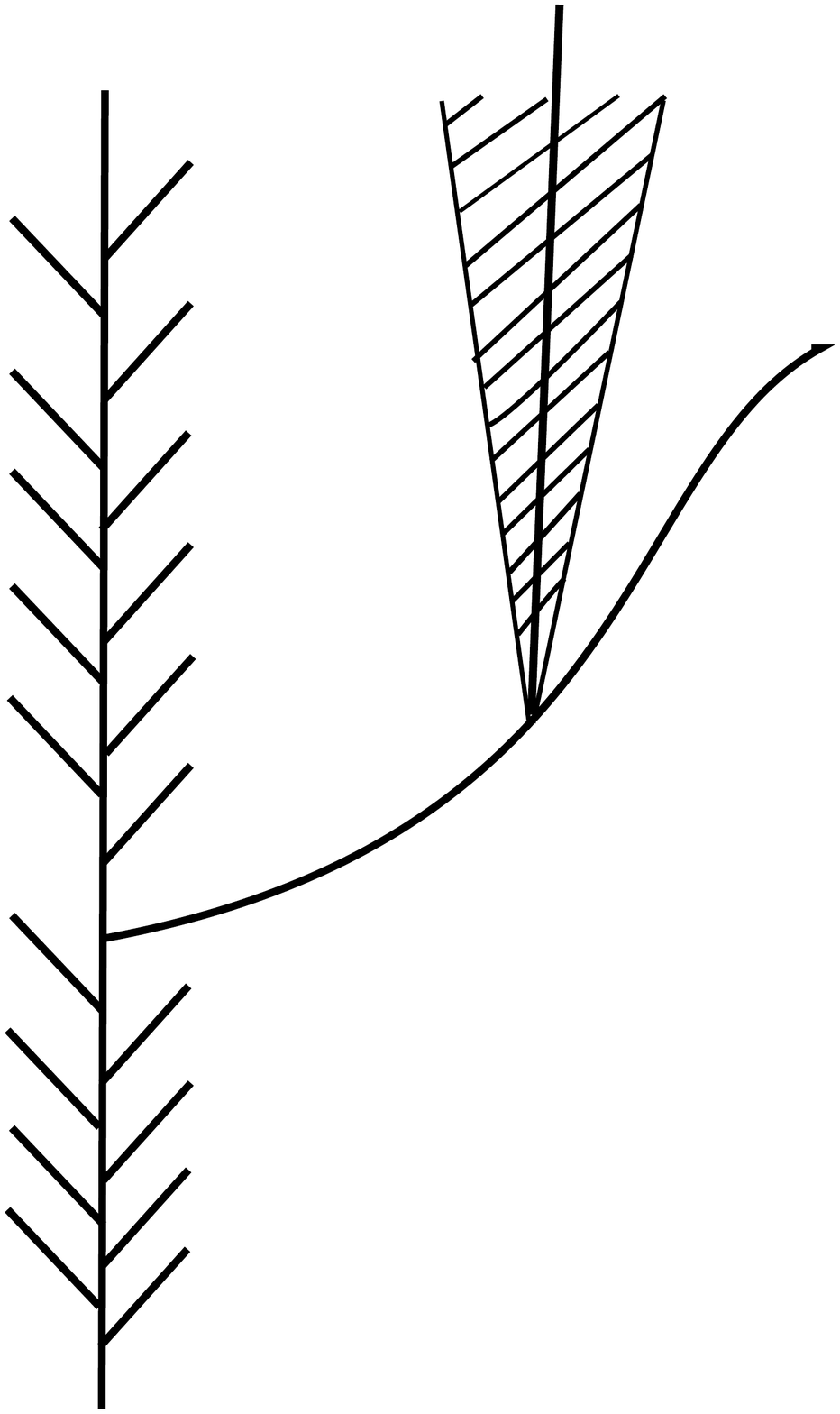 
\caption[Caption for LOF]{The case $i=2$ \cite[Page 68]{Droste85}}
\end{figure}

%

 Let $C_{\beta}$ be the disjoint union of the $r-1$ (or $r$ if $r$ is infinite) maximal subchains constructed.
We obtain $Y_{i+1}$ by adjoining at each $\beta\in Z_{i-1}$ the set $C_{\beta}$, that is, let
\[ Z_{i}=\bigsqcup_{\beta\in Z_{i-1}} C_{\beta} \quad \text{and} \quad Y_{i+1}=Y_i \sqcup Z_{i}.
\] 
Then conditions $(\mathrm{i}),(\mathrm{ii})$ and $(\mathrm{iii})$ are shown to hold, and $ \bigcup_{i\in \mathbb{N}} Y_i=T_r$  as desired.

\bigskip 

We can use this construction to obtain automorphisms of $T_r$. Suppose we also reconstruct $T_r$ from $\alpha'\in T_r$ via sets $Y_i',Z_i', C_{\beta}'$ (so that $ \bigcup_{i\in \mathbb{N}} Y_i'=T_r$).  
Let $\pi_1:Y_1\rightarrow Y_1'$ be an isomorphism such that $\alpha\pi_1=\alpha'$ (such an isomorphism exists as maximal chains are isomorphic to $\mathbb{Q}$, and $\mathbb{Q}$ is homogeneous). Suppose the isomorphism  $\pi_i:Y_i\rightarrow Y_i'$ has already been defined for some $i\in \mathbb{N}$. Then we may extend $\pi_i$ to $\pi_{i+1}:Y_{i+1}\rightarrow Y_{i+1}'$ as follows. For each $\beta\in Z_{i-1}$ the posets $C_{\beta}$ and $C_{\beta\pi_i}$ are both disjoint unions of the same number of copies of $\mathbb{Q}$, and are thus isomorphic (as posets). Let $\phi_{\beta}:C_{\beta}\rightarrow C_{\beta\pi_i}$ be an isomorphism, and let 
\[ \pi_{i+1}=\pi_i\sqcup  \bigsqcup_{\beta\in Z_{i-1}} \phi_{\beta}: Y_{i+1}\rightarrow Y_{i+1}'.
\] 
Then $\pi_{i+1}$ is an isomorphism, and so $\pi=\bigcup_{i\in \mathbb{N}} \pi_i$ is an automorphism of $T_r$. 

Before giving a full classification of homogeneous image-trivial normal bands, it is worth giving a simplified isomorphism theorem, which follows easily from Proposition \ref{iso normal}.  

\begin{corollary}\label{iso trivial} Let $B=[Y;B_{\alpha};\psi_{\alpha,\beta};\epsilon_{\alpha,\beta}]$ and $B'=[Y';B_{\alpha'}';\psi'_{\alpha',\beta'};\epsilon_{\alpha',\beta'}']$ be a pair of image-trivial normal bands. Let $\pi:Y\rightarrow Y'$ be an isomorphism, and for each $\alpha\in Y$ let $\theta_{\alpha}:B_{\alpha}\rightarrow B_{\alpha\pi}'$ be an isomorphism. Then $\bigcup_{\alpha\in Y}\theta_{\alpha}$ is an isomorphism from $B$ into $B'$ if and only if $\epsilon_{\alpha,\beta}\theta_{\beta}= \epsilon_{\alpha\pi,\beta\pi}'$ for each $\alpha\geq \beta$ in $Y$. 
\end{corollary} 

A subsemigroup $A$ of an image-trivial normal band $[Y;B_{\alpha};\psi_{\alpha,\beta};\epsilon_{\alpha,\beta}]$ is called a \textit{maximal chain} if $A$ is a semilattice and \text{supp}$(A)$ is a maximal chain in $Y$. Note that if $Y$ is a homogeneous semilinear order and $A$ is a maximal chain in $B$ then 
\[ A=\bigcup_{\alpha>\beta \text{ in } \text{supp} (A)} \epsilon_{\alpha,\beta} \cong \mathbb{Q}
\] 
We shall use the construction of $T_r$ above to prove the following: 

\begin{proposition}\label{1 trans} Let $B=[T_r;B_{\alpha};\psi_{\alpha,\beta};\epsilon_{\alpha,\beta}]$ and $B'=[T_r;B_{\alpha'}';\psi'_{\alpha',\beta'};\epsilon_{\alpha',\beta'}']$ be a pair of image-trivial normal bands with ramification order $k$ such that there exists $n,m\in \mathbb{N}^*$ with $B_{\alpha}\cong B_{\alpha'}'\cong B_{n,m}$ for all $\alpha,\alpha'\in T_r$. Let $e\in B$ and $f\in B'$, and consider a pair of sub-rectangular bands $M\subseteq B$ and $N \subseteq B'$ with $M>e$ and $N>f$. Then for any isomorphism $\Phi:M\cup \{e\} \rightarrow N\cup \{f\}$, there exists an isomorphism $\theta:B\rightarrow B'$ extending $\Phi$. Consequently, $B$ is 1-homogeneous. 
\end{proposition} 

\begin{proof} We may assume $r>1$, else $B$ and $B'$ are isomorphic to $\mathbb{Q}$. 
Let $e_{\sigma},e_{\sigma'}\in B$,  $M$ a rectangular subband of $B_{\alpha}$, and $N$ a rectangular subband of  $B_{\delta}$, with $M>e_{\sigma}$ and $N>e_{\sigma'}$ for some $\sigma,\sigma',\alpha,\delta\in T_r$.
 Consider an isomorphism
\[ \Phi:M\cup \{e_{\sigma}\}\rightarrow N\cup \{e_{\sigma'}\},
\] 
so that $M\Phi=N$ and $e_{\sigma}\Phi=e_{\sigma'}$. By Corollary \ref{rec bands homog}, we may extend $\Phi|_M$ to an isomorphism $\Phi':B_{\alpha}\rightarrow B_{\delta}'$.
 Fix some $e_{\alpha}\in M$ and let $e_{\alpha}\Phi=e_{\delta}$. Consider a pair of enumerations $\{a_i:i\in \mathbb{N}\}$ and $\{b_i:i\in \mathbb{N}\}$ of $T_r$ (where $r=knm$) such that $\alpha=a_1$ and $\delta=b_1$. Let $A$ be a maximal chain in $B$ such that $e_{\sigma},e_{\alpha}\in A$, and let $Y_1=Z_0=\text{supp}(A)$ (so $Y_1\cong A$). Similarly obtain $e_{\sigma'},e_{\delta}\in \hat{A}$ and $\hat{Y}_1=\hat{Z}_0=\text{supp}(\hat{A})$ (so $\hat{Y_1}\cong \hat{A}$). Take an isomorphism $\pi_1:Y_1\rightarrow \hat{Y}_1$ such that $\sigma\pi_1=\sigma'$ and $\alpha\pi_1=\delta$ (again this is possible as $Y_1$ and $\hat{Y}_1$ are isomorphic to $\mathbb{Q}$). For each $\beta\in  Y_1\setminus\{\alpha\}$, take any isomorphism $\theta_{\beta}:B_{\beta}\rightarrow B_{\beta\pi_1}'$ such that $(B_{\beta}\cap A)\theta_{\beta}=B_{\beta\pi_1}' \cap \hat{A}$ (such an isomorphism exists by Corollary \ref{rec bands homog}), and let $\theta_{\alpha}=\Phi'$. Letting $D_1=[Y_1;B_{\alpha};\psi_{\alpha,\beta};\epsilon_{\alpha,\beta}]$ and $\hat{D}_1= [\hat{Y}_1;B_{\alpha'}';\psi'_{\alpha',\beta'};\epsilon_{\alpha',\beta'}']$, the map
\[ \theta_1=[\theta_{\beta},\pi_1]_{\beta\in Y_1}:D_1 \rightarrow \hat{D}_1
\] 
 is an isomorphism by Corollary \ref{iso trivial}, since $B_{\beta}\cap A=\{\epsilon_{\gamma,\beta}\}$ for all $\gamma>\beta$ in $Y_1$, and $B_{\beta\pi_1}\cap \hat{A}=\{\epsilon_{\gamma\pi_1,\beta\pi_1}\}$ for all $\gamma\pi_1>\beta\pi_1$ in $\hat{Y_1}$. 
 
Suppose for some $i\in \mathbb{N}$ the semilattices $Y_j,\hat{Y}_j$, posets $Z_{j-1},\hat{Z}_{j-1}$, bands 
 \[ D_j=[Y_j;B_{\alpha};\psi_{\alpha,\beta};\epsilon_{\alpha,\beta}], \\\ \hat{D}_j=[\hat{Y}_j;B_{\alpha'}';\psi'_{\alpha',\beta'};\epsilon_{\alpha',\beta'}']
 \]  and isomorphisms $\pi_j:Y_j\rightarrow \hat{Y}_j$, $\theta_j=[\theta_{\alpha},\pi_j]_{\alpha\in Y_j}:D_j\rightarrow \hat{D}_j$ have already been defined for each $j\leq i$, and are such that $Y_j,Z_{j-1}$ and $\hat{Y}_j,\hat{Z}_{j-1}$ satisfy conditions $(\mathrm{i}),(\mathrm{ii})$ and $(\mathrm{iii})$. As in the semilattice construction, if $a_{i+1}\not\in Y_i$ then we may fix $z\in Z_{i-1}$ such that $a_{i+1}$ belongs to some cone of $z$ which is disjoint to $Y_i$, and let $B_{a_{i+1}}\psi_{a_{i+1},z}=e_z$.
 
Consider the subset $X_{i-1}=\bigcup_{\gamma\in Z_{i-1}} B_{\gamma}$ of $B$. For each $e_{\beta}\in X_i$ (so $\beta\in Z_{i-1}$), take a maximal subchain of each cone $C\in C(e_{\beta})$ such that $Y_i\cap C=\emptyset$. If $e_{\beta}=\epsilon_{y,\beta}$ for (any) $y$ in the chain $\{y\in Y_i:\beta<y\}$, then by condition $(\mathrm{ii})$ precisely one cone will intersect $Y_i$ non-trivially. Otherwise, all cones of $e_{\beta}$ intersects $Y_i$ trivially. Moreover, if $a_{i+1}\not\in Y_i$ and $\beta=z$, we also require the maximal subchain of a cone of $C(e_z)$ to contain $a_{i+1}$. 
  
 Let $C_{e_\beta}$ be the disjoint union of the $k$ (or $k-1$ if $e_{\beta}=\epsilon_{y,\beta}$ for some $y\in Y_i$, and $k$ is finite) maximal subchains and let 
 \begin{align*} &  Z_i=\bigsqcup_{e_{\beta}\in X_{i-1}} C_{e_{\beta}}.
\end{align*} 
Let $Y_{i+1}=Y_i \sqcup Z_{i}$, and note that $\gamma \leq \gamma'$ for $\gamma,\gamma' \in Y_{i+1}$ if and only if  
\begin{align*} & \text{either } \gamma,\gamma' \in Y_i \text{ and } \gamma \leq \gamma' \text{ in } Y_i; \\
& \text{or } \gamma,\gamma' \in C_{e_{\beta}} \text{ for some } e_\beta\in X_{i-1} \text{ and } \gamma \leq \gamma' \text{ in } C_{e_\beta};\\
& \text{or } \gamma\in Y_i, \gamma' \in C_{e_\beta} \text{ for some } e_\beta\in X_{i-1} \text{ and } \beta \geq \gamma \text{ in } Y_i.
\end{align*}  Similarly obtain $ \hat{C}_{e_{\beta'}}$, $\hat{Z}_i$ and $\hat{Y}_{i+1}$, noting that as $B$ has ramification order $k$ the set  $ \hat{C}_{e_{\beta'}}$ will also be formed from $k$ (or $k-1$ if $e_{\beta'}=\epsilon'_{y',\beta'}$ for some $y'\in \hat{Y}_i$, and $k$ is finite)  maximum subchains. Let $D_{i+1}=[Y_{i+1};B_{\alpha};\psi_{\alpha,\beta};\epsilon_{\alpha,\beta}]$ and $\hat{D}_{i+1}=[\hat{Y}_{i+1};B_{\alpha'}';\psi'_{\alpha',\beta'};\epsilon_{\alpha',\beta'}']$. 
 
Recall that $C(\beta)=\bigcup_{e_{\beta}\in B_{\beta}}C(e_{\beta})$ for all $\beta\in T_r$. Hence as $\bigcup_{e_{\beta}\in B_{\beta}} C_{e_{\beta}}$ is a set of maximal subchains of the $r-1$ (or $r$ if $r$ is infinite) cones of $C(\beta)$ which intersect $Y_i$ non-trivially, it follows that conditions $(\mathrm{i}),(\mathrm{ii})$ and $(\mathrm{iii})$ are satisfied and $\bigcup_{i\in \mathbb{N}} Y_i = T_r$ (similarly for $\hat{Y}_i$). Consequently, $B=\bigcup_{i\in \mathbb{N}} D_i$ and $B'=\bigcup_{i\in \mathbb{N}} \hat{D}_i$.

For each $e_{\beta}\in X_{i-1}$, let $\sigma_{e_{\beta}}:C_{e_{\beta}}\rightarrow \hat{C}_{e_{\beta}\theta_i}$ be an isomorphism (as posets), and let 
\[ \pi_{i+1}=\pi_i\sqcup  \bigsqcup_{e_{\beta}\in X_{i-1}} \sigma_{e_\beta}: Y_{i+1}\rightarrow \hat{Y}_{i+1}.
\] 
By the order on $Y_{i+1}$ defined above, it is easily shown that $\pi_{i+1}$ is an isomorphism, and so $\pi=\bigcup_{i\in \mathbb{N}} \pi_i$ is an automorphism of $T_r$. For each $\gamma\in C_{e_{\beta}}$ ($e_{\beta}\in X_{i-1}$), let $\theta_{\gamma}:B_{\gamma}\rightarrow B_{\gamma\pi_{i+1}}$ be an isomorphism such that $\epsilon_{\gamma',\gamma}\theta_{\gamma}=\epsilon_{\gamma'\pi_{i+1},\gamma\pi_{i+1}}$ for (any) $\gamma'\in C_{e_{\beta}}$ with $\gamma'>\gamma$. We claim that the map 
\[ \theta_{i+1}=\theta_i\sqcup \bigsqcup_{\substack{\gamma\in C_{e_{\beta}} \\ e_{\beta}\in X_{i-1}}} \theta_{\gamma}= [\theta_{\alpha},\pi_{i+1}]_{\alpha\in Y_{i+1}}:D_{i+1}\rightarrow \hat{D}_{i+1}
\] 
is an isomorphism. Suppose $\gamma,\gamma'\in Y_{i+1}$ are such that $\gamma\leq \gamma'$. If $\gamma,\gamma'\in Y_{i}$, then as $\theta_i$ preserves the images of the connecting morphisms from $Y_i$, we have 
\[ \epsilon_{\gamma',\gamma}\theta_{i+1}=\epsilon_{\gamma',\gamma}\theta_i=\epsilon_{\gamma'\pi_i,\gamma\pi_i}= \epsilon_{\gamma'\pi_{i+1},\gamma\pi_{i+1}}.
\]
Similarly, if $\gamma,\gamma'\in C_{e_{\beta}}$ for some $e_{\beta}\in X_{i-1}$ then by construction $\epsilon_{\gamma',\gamma}\theta_{\gamma}=\epsilon_{\gamma'\pi_{i+1},\gamma\pi_{i+1}}$. Finally, if $\gamma\in Y_i$, $\gamma'\in C_{e_{\beta}}$ for some $e_{\beta}\in X_{i-1}$ and $\beta\geq \gamma$, then
\[ \epsilon_{\gamma',\gamma}\theta_{i+1}=\epsilon_{\beta,\gamma}\theta_{i+1}=\epsilon_{\beta,\gamma}\theta_{i}=\epsilon_{\beta\pi_{i},\gamma\pi_{i}} = \epsilon_{\beta\pi_{i+1},\gamma\pi_{i+1}}=\epsilon_{\gamma'\pi_{i+1},\gamma\pi_{i+1}}
\] 
by \eqref{eq trivial} since $\gamma'>\beta\geq \gamma$ and $\gamma'\pi_{i+1}>\beta\pi_{i+1}\geq \gamma\pi_{i+1}$. The claim then follows by Corollary \ref{iso trivial}. 
Hence $\theta=\bigcup_{i\in \mathbb{N}}\theta_i$ is an isomorphism from $B$ to $B'$ which extends $\Phi$. Taking $B=B'$ shows that $B$ is 1-homogeneous.
\end{proof} 

Consequently, for each collection $r,k,n,m\in \mathbb{N}^*$ such that $r=knm$, there exists a unique, up to isomorphism, image-trivial normal band  $B=[T_r;B_{\alpha};\psi_{\alpha,\beta};\epsilon_{\alpha,\beta}]$ with ramification order $k$ and $B_{\alpha}\cong B_{n,m}$ for all $\alpha\in T_r$. We shall denote such a band $T_{n,m,k}$, where $r=nmk$.

\begin{proposition}\label{T homog} A homogeneous image-trivial normal band is isomorphic to $T_{n,m,k}$ for some $n,m,k\in \mathbb{N}^*$. Conversely, every band $T_{n,m,k}$ is homogeneous. 
\end{proposition} 

\begin{proof} By Lemma \ref{trivial not universal} an image-trivial homogeneous normal band has a semilinear structure semilattice, and has a ramification order by 1-homogeneity. Hence every homogeneous image-trivial normal band is isomorphic to some $T_{n,m,k}$ by the preceding results. We now prove that $T_{n,m,k}=[T_r;B_{\alpha};\psi_{\alpha,\beta};\epsilon_{\alpha,\beta}]$ is homogeneous. Since $T_{n,m,k}$ is 1-homogeneous by the proposition above, we may proceed by induction, by supposing all isomorphisms between finite subbands of size $j-1$ extend to an automorphism of $B$. Let $M=[Z_1;M_{\alpha};\psi_{\alpha,\beta}^M], N=[Z_2;N_{\alpha};\psi_{\alpha,\beta}^N]$ be a pair of finite subbands of $B$ of size $j$, and $\theta=[\theta_{\alpha},{\pi}]_{\alpha\in Z_1}$ an isomorphism from $M$ to $N$. By Proposition \ref{1 trans} we may assume that $Z_1$ and $Z_2$ are non-trivial (so $N,M$ are not rectangular bands). Let $\delta$ be maximal in $Z_1$, and $\delta\pi=\delta'$. Then by the inductive hypothesis the isomorphism ${\theta}|_{M\setminus M_{\delta}}: M\setminus M_{\delta}\rightarrow N \setminus N_{\delta'}$ extends to an automorphism $\theta^*=[\theta^*_{\alpha},\pi^*]_{\alpha\in T_r}$ of $B$. 

Since $Z_1$ is a finite semilinear order, there exists a unique $\beta\in Z_1$ covered by $\delta$. 
As $\theta$ is an isomorphism, it follows from Corollary \ref{iso trivial} that $\epsilon_{\delta,\beta}\theta_{\beta}^*=\epsilon_{\delta,\beta}\theta=\epsilon_{\delta',\beta\pi}=\epsilon_{\delta',\beta\pi^*}$. For each $e_{\tau}\in B$, consider the subsemilattice of $T_r$ given by 
\[ [e_{\tau}]:=\{\gamma\in T_r:B_{\gamma}\psi_{\gamma,\tau}=e_{\tau}\}=\text{supp}\{e\in B:e>e_{\tau}\}. 
\] 
Note that $[e_{\tau}]$ is the union of the cones of $e_{\tau}$, and so $T_r\setminus [e_{\tau}]$ forms a subsemilattice of $T_r$. Then $\hat{\pi}^*=\pi^*|_{T_r \setminus [\epsilon_{\delta,\beta}]}:T_r \setminus [\epsilon_{\delta,\beta}]\rightarrow T_r \setminus [\epsilon_{\delta',\beta\pi^*}]$ is an isomorphism, and we now aim to extend the isomorphism 
 \[ [\theta_{\gamma}^*, \hat{\pi}^*]_{\gamma\in T_r \setminus [\epsilon_{\delta,\beta}]} : \bigcup_{\gamma\in T_r \setminus [\epsilon_{\delta,\beta}]}B_{\gamma}\rightarrow \bigcup_{\gamma\in T_r \setminus [\epsilon_{\delta',\beta\pi^*}]}B_{\gamma}
 \]
 to an automorphism of $B$ which also extends ${\theta}$. 
 
Since $\theta$ maps $M_{\delta} \cup \{\epsilon_{\delta,\beta}\}$ to $N_{\delta\pi}\cup \{\epsilon_{\delta',\beta\pi}\}$, we may extend the isomorphism $\theta|_{M_{\delta} \cup \{\epsilon_{\delta,\beta}\}}$ to an automorphism $\bar{\theta}^*=[\bar{\theta}_{\gamma}^*,\bar{\pi}^*]_{\gamma\in T_r}$ of $B$ by Proposition \ref{1 trans}. Then as $\beta\hat{\pi}^* = \beta \bar{\pi}^*$, the bijection $\tilde{\pi}=\hat{\pi}^*|_{T_r \setminus [\epsilon_{\delta,\beta}]} \sqcup \bar{\pi}^*|_{[\epsilon_{\delta,\beta}]}$ is an automorphism of $T_r$. 

 We claim that $\mathbf{\tilde{\theta}}=[\mathbf{\tilde{\theta}}_{\gamma},\mathbf{\tilde{\pi}}]_{\gamma\in T_r}$, where 
 \begin{align*}
 \mathbf{\tilde{\theta}}_{\gamma} =
  \begin{cases}
   \theta_{\gamma}^* & \text{if } \gamma\in T_r \setminus [\epsilon_{\delta,\beta}],\\ 
  \bar{\theta}^*_{\gamma} & \text{if } \gamma\in [\epsilon_{\delta,\beta}],   \end{cases} 
\end{align*}
 is an automorphism of $T_{n,m,k}$. Indeed, by Corollary \ref{iso trivial} it is sufficient to prove that $\epsilon_{\gamma',\gamma} \mathbf{\tilde{\theta}} = \epsilon_{\gamma'\mathbf{\tilde{\pi}},\gamma\mathbf{\tilde{\pi}}}$ for any $\gamma'\geq \gamma$ in $T_r$. Note if $\gamma\in [\epsilon_{\delta,\beta}]$ then $\gamma'\geq\gamma > \beta$ and so $\gamma'\in  [\epsilon_{\delta,\beta}]$ by \eqref{eq trivial}. Hence, as $\theta^*$ and $\bar{\theta}^*$ are automorphisms of $B$ (and by the construction of $\mathbf{\tilde{\theta}}$), we only need consider the case where  $\gamma'\in  [\epsilon_{\delta,\beta}]$ and $\gamma\in T_r\setminus [\epsilon_{\delta,\beta}]$. If $\gamma\neq \beta$ then $\gamma'> \beta > \gamma$, so $\epsilon_{\gamma',\gamma}=\epsilon_{\beta,\gamma}$ by \eqref{eq trivial}, and so as $\beta,\gamma\in T_r\setminus [\epsilon_{\delta,\beta}]$, 
 \begin{align*} \epsilon_{\gamma',\gamma}\mathbf{\tilde{\theta}}= \epsilon_{\beta,\gamma}\mathbf{\tilde{\theta}}= \epsilon_{\beta,\gamma}\theta_{\gamma}^*=\epsilon_{\beta\hat{\pi}^*,\gamma\hat{\pi}^*} = \epsilon_{\beta\mathbf{\tilde{\pi}},\gamma\mathbf{\tilde{\pi}}} = \epsilon_{\gamma'\mathbf{\tilde{\pi}},\gamma\mathbf{\tilde{\pi}}}
 \end{align*}  
 with the final equality holding since $\gamma'\mathbf{\tilde{\pi}} > \beta\mathbf{\tilde{\pi}} > \gamma\mathbf{\tilde{\pi}}$. Finally, if $\gamma=\beta$ then 
 \[\epsilon_{\gamma',\gamma}\mathbf{\tilde{\theta}}= \epsilon_{\gamma',\beta}\mathbf{\tilde{\theta}}_{\beta}=\epsilon_{\gamma',\beta}\theta_{\beta}^* = \epsilon_{\gamma',\beta}\theta^*_{\beta}=  \epsilon_{\delta,\beta}\theta^*_{\beta} = \epsilon_{\delta',\beta\hat{\pi}^*}= \epsilon_{\gamma'\bar{\pi}^*,\beta\hat{\pi}^*}=\epsilon_{\gamma' \mathbf{\tilde{\pi}}, \beta \mathbf{\tilde{\pi}}} 
 \] 
 since $\gamma'\bar{\pi}^*\in [\epsilon_{\delta',\beta\pi^*}]= [\epsilon_{\delta',\beta\hat{\pi}^*}]$. 
 Thus $\mathbf{\tilde{\theta}}$ is indeed an automorphism of $B$, and extends $\theta$ by construction. 
\end{proof} 
It is worth noting that the spined product of a left and right homogeneous image-trivial normal band need not be homogeneous. For example suppose $B=T_{2,1,1}\bowtie T_{1,2,1}$ is homogeneous, and thus isomorphic to some $T_{n,m,k}$. Then $n=2,m=2$ and as $B$ has structure semilattice $T_2$, so must $T_{2,2,k}$, and so $2.2.k=4k=2$, contradicting $k\in \mathbb{N}^*$. Our aim is now to prove that the converse holds, that is, if an image-trivial normal band $L\bowtie R$ is homogeneous, then so are $L$ and $R$.

\begin{corollary}\label{trivial splits} Let $B=L\bowtie R$ be a homogeneous normal band such that $L$ is image-trivial. Then $L$ is homogeneous (dually for $R$).  
\end{corollary} 

\begin{proof} Let $B=[Y;L_{\alpha};\psi_{\alpha,\beta}^l;\epsilon_{\alpha,\beta}^l]\bowtie [Y;R_{\alpha};\psi_{\alpha,\beta}^r]$ be homogeneous.
  Then by Corollary \ref{basics} ($\mathrm{iii}$)  there exists $n\in \mathbb{N}^*$ such that $L_{\alpha}\cong B_{n,1}$ for all $\alpha\in Y$, and by Lemma \ref{trivial not universal} we may assume $Y=T_r$, where $r=nk$ for some $k\in \mathbb{N}^*$. Moreover, $L$ has a ramification order, since if $l_{\alpha},k_{\beta}\in L$, then by fixing any $r_{\alpha}\in R_{\alpha}$, $s_{\beta}\in R_{\beta}$, there exists an automorphism $\theta^l\bowtie \theta^r$ of $B$ sending $(l_{\alpha},r_{\alpha})$ to $(k_{\beta},s_{\beta})$ by Corollary \ref{basics} ($\mathrm{i}$). 
  In particular, $l_{\alpha}\theta^l=k_{\beta}$ and so $|C(l_{\alpha})|= |C(k_{\beta})|=k$. Hence $L\cong T_{n,1,k}$ by Proposition \ref{1 trans}, and is thus homogeneous. 
\end{proof}

\subsection{Surjective normal bands} \label{sec:surj}

We now study the homogeneity of surjective normal bands. 

\begin{lemma} \label{embed Y} Let $B=[Y;B_{\alpha};\psi_{\alpha,\beta}]$ be a homogeneous surjective normal band. Then for any finite subsemilattice $Z$ of $Y$, there exists a subband $A=[Z;\{e_{\alpha}\};\psi_{\alpha,\beta}^A]$ of $B$ isomorphic to $Z$. 
\end{lemma} 

\begin{proof} Suppose first that $Y$ is a linear or semilinear order.  The result trivially holds for the case where $|Z|=1$ by taking $A$ to be a trivial subband. Proceed by induction by assuming that the result holds for all subsemilattices of size $n-1$, and let $Z$ be a subsemilattice of $Y$ of size $n\in \mathbb{N}$. Let $\delta$ be maximal in $Z$, so $Z'=Z\setminus \{\delta\}$ is a subsemilattice of $Y$ of size $n-1$. By the inductive hypothesis, there exists a subband $A'=[Z';\{e_{\alpha}\};\psi_{\alpha,\beta}^{A'}]$ and an isomorphism $\phi:A'\rightarrow  Z'$. Since $Y$ is linearly or semilinearly ordered and $Z$ is finite, there is a unique $\beta\in Z'$ such that $\delta$ covers $\beta$. Let $\{e_{\beta}\}= A'\cap B_{\beta}$. Since $\psi_{\delta,\beta}$ is surjective, there exists $e_{\delta}\in B_{\delta}$ such that $e_{\delta}\psi_{\delta,\beta}=e_{\beta}$. Let $\phi'$ be the map from $A'\cup \{e_{\delta}\}$ to $Z$ given by $A'\phi'=A'\phi$ and $e_{\delta}\phi'=\delta$. Then $\phi'$ is clearly an isomorphism, and the inductive step is complete. 

Suppose instead that $Y$ contains a diamond $\beta<\{\tau,\gamma\}<\sigma$. We claim that any pair $\alpha,\delta\in Y$ with $\alpha \perp \delta$ has an upper cover. Let $e_{\alpha\delta}\in B_{\alpha\delta}$ be fixed. Since the connecting morphisms are surjective, $e_{\alpha\delta}=e_{\alpha}\psi_{\alpha,\alpha\delta}=e_{\delta}\psi_{\delta,\alpha\delta}$ for some $e_{\alpha}\in B_{\alpha}$ and $e_{\delta}\in B_{\delta}$. Similarly for $\tau$ and $\gamma$ we obtain $e_{\beta}=e_{\tau}\psi_{\tau,\beta}=e_{\gamma}\psi_{\gamma,\beta}$, and the claim follows by extending the isomorphism from $\{e_{\beta},e_{\tau},e_{\gamma}\}$ to $\{e_{\alpha\delta},e_{\alpha},e_{\delta}\}$ to an automorphism of $B$. Hence, by a simple inductive argument, every finite subsemilattice of $Y$ has an upperbound. Let $Z$ be a finite subsemilattice of $Y$ and $\alpha\in Y$ be such that $\alpha>Z$. Then for any $e_{\alpha}\in B_{\alpha}$, 
\[ \{e_\alpha\psi_{\alpha,\beta}: \beta\in Z\}\cong Z, 
\] 
as required. 
\end{proof} 

\begin{corollary}\label{Y homog surj} Let $B=[Y;B_{\alpha};\psi_{\alpha,\beta}]$ be a homogeneous  normal band. Then $Y$ is homogeneous. 
\end{corollary} 

\begin{proof}  Suppose first that $B$ is a surjective normal band. Let $\pi:Z\rightarrow Z'$ be an isomorphism between finite subsemilattices of $Y$. By Lemma \ref{embed Y}, there exists subbands $A=[Z;\{e_{\alpha}\};\psi^A_{\alpha,\beta}]$ and $A'=[Z';\{e_{\alpha'}\};\psi^{A'}_{\alpha',\beta'}]$ isomorphic to $Z$ and $Z'$, respectively. Hence $[\theta_{\alpha},\pi]_{\alpha\in Z}$ is an isomorphism from $A$ to $A'$, where $\theta_{\alpha}$ maps $e_{\alpha}$ to $e_{\alpha\pi}$, and the result follows by the homogeneity of $B$. 
Now let $B=L\bowtie R$ be an arbitrary homogeneous normal band. By Lemma \ref{trivial or surj}, $L$ and $R$ are either image-trivial or surjective normal bands. If both $L$ and $R$ are surjective, then clearly so too is $B$, and so $Y$ is homogeneous by the first part. Otherwise, $Y$ is homogeneous by Lemma \ref{trivial not universal}. 
\end{proof} 

\begin{corollary}\label{surj split} Let $B=L\bowtie R$ be a homogeneous surjective normal band. Then $L$ and $R$ are homogeneous. 
\end{corollary} 

\begin{proof}
Since $B$ is a surjective normal band, the normal bands $L=[Y;L_{\alpha};\psi_{\alpha,\beta}^l]$ and $R=[Y;R_{\alpha};\psi_{\alpha,\beta}^r]$ are also surjective. 
 Let $L_i=[Z_i;L_{\alpha}^i;\psi_{\alpha,\beta}^{L_i}]$ ($i=1,2$) be a pair of finite subbands of $L$ and $\theta^l=[\theta_{\alpha}^l,\pi]_{\alpha\in Z_1}$ an isomorphism from $L_1$ to $L_2$. By Lemma \ref{embed Y}, there exists subbands $A_i=\{(l_{\alpha}^i,r_{\alpha}^i):\alpha\in Z_i\}$ of $B$ isomorphic to $Z_i$. Hence $R_i=\{r_{\alpha}^i:\alpha\in Z_i\}$ is a subband of $R$ isomorphic to $Z_i$ for each $i$, and the map $\theta^r:R_1\rightarrow R_2$ given by $r_{\alpha}^1\theta^r=r_{\alpha\pi}^2$ is an isomorphism.
  By Proposition \ref{iso regular}, $\theta^l\bowtie \theta^r$ is an isomorphism from $L_1\bowtie R_1$ to $L_2 \bowtie R_2$, which we may extend to an automorphism $\hat{\theta}=\hat{\theta}^l\bowtie \hat{\theta}^r$ of $B$. Then $\hat{\theta}^l$ extends $\theta^l$ and so $L$ is homogeneous. Dually for $R$. 
\end{proof}

Consider now the case where $B=[Y;B_{\alpha};\psi_{\alpha,\beta}]$ is such that there exists $\alpha>\beta$ in $Y$ with $\psi_{\alpha,\beta}$ an isomorphism. Then by Lemma \ref{image same} every connecting morphism is an isomorphism. We extend our notation by defining the morphism $\psi_{\alpha,\beta}$ by 
\[ {\psi}_{\alpha,\beta}=\psi_{\alpha,\alpha\beta}(\psi_{\beta,\alpha\beta})^{-1}
\] 
 for any $\alpha,\beta\in Y$. Note that if $\alpha\geq \beta$ then $\psi_{\alpha,\beta}$ is the same morphism as before, and it is a simple exercise to show that ${\psi}_{\alpha,\beta}{\psi}_{\beta,\gamma}={\psi}_{\alpha,\gamma}$ for all $\alpha,\beta,\gamma\in Y$ by the transitivity of the connecting morphisms. Moreover, for each $\alpha,\beta\in Y$ we have $\psi_{\alpha,\beta}^{-1} = \psi_{\beta,\alpha}$. 

\begin{lemma} \label{iso surj} Let $B=[Y;B_{\alpha};\psi_{\alpha,\beta}]$ be a normal band such that each $\psi_{\alpha,\beta}$ is an isomorphism. Let $\pi \in \text{Aut }(Y)$ and, for a fixed $\alpha^*\in Y$, let $\theta_{\alpha^*}:B_{\alpha^*}\rightarrow B_{\alpha^*\pi}$ be an isomorphism. For each $\delta \in Y$, let $\theta_{\delta}:B_{\delta}\rightarrow B_{\delta\pi}$ be given by 
\begin{equation} \label{iso iso} {\theta}_{\delta}= {\psi}_{\delta,\alpha^*} \, {\theta}_{\alpha^*} \, {\psi}_{\alpha^*\pi,\delta \pi}. 
\end{equation}  
Then $\theta=[\theta_{\alpha},\pi]_{\alpha\in Y}$ is an automorphism of $B$. Conversely, every automorphism of $B$ can be so constructed. 
\end{lemma} 

\begin{proof}
Let $\pi$ and $\theta_{\delta}$ be defined as in the statement of the lemma for each $\delta\in Y$. If $\delta\geq \gamma$ then 
\begin{align*}
  \psi_{\delta,\gamma} \, {\theta}_{\gamma} & =   \psi_{\delta,\gamma} \,  {\psi}_{\gamma,\alpha^*} \, {\theta}_{\alpha^*} \, {\psi}_{\alpha^*\pi,\gamma\pi} \\
 & =    {\psi}_{\delta,\alpha^*} \, {\theta}_{\alpha^*} \, {\psi}_{\alpha^*\pi,\delta\pi} \,  {\psi}_{\delta\pi,\gamma\pi} \\
 & = {\theta}_{\delta} \, \psi_{\delta\pi,\gamma\pi},
 \end{align*}
where the penultimate step follows from $\pi$ being an automorphism, so that $\delta\pi\geq \gamma\pi$. Hence $[\delta,\gamma;\delta\pi,\gamma\pi]$ commutes, and so $\theta$ is an automorphism of $B$. 

Conversely, suppose $\theta=[\theta_{\alpha},\pi]_{\alpha\in Y}$ is an automorphism of $B$,
and fix any $\alpha^*\in Y$. Then for each $\delta\in Y$, since the connecting morphisms are
 isomorphisms and both the diagrams $[\alpha^*,\alpha^*\delta;\alpha^*\pi,(\alpha^*\delta)\pi]$ and $
 [\delta,\alpha^*\delta;\delta\pi,(\alpha^*\delta)\pi]$ commute, it follows that $
 \theta_{\delta}$ has the form of \eqref{iso iso}.  
\end{proof}

\begin{proposition}\label{iso hom} Let $B=[Y;B_{\alpha};\psi_{\alpha,\beta}]$ be a normal band such that each connecting morphism is an isomorphism and $Y$ is a homogeneous semilattice. Then $B$ is structure-homogeneous. 
\end{proposition} 

\begin{proof} 
Consider a pair of finite subbands of $B$ 
\begin{align*}  A &=[Z;A_{\alpha}; {\psi}_{\alpha,\beta}^{A}] \\
 A' & =  [Z';A_{\alpha'}';{\psi}_{\alpha',\beta'}^{A'}] 
\end{align*}
where $\psi_{\alpha,\beta}^A$ and $\psi_{\alpha',\beta'}^{A'}$, being restrictions of isomorphisms, are embeddings. Let $\theta=[\theta_{\alpha},\pi]_{\alpha\in Z}$ be an isomorphism from $A$ to $A'$, and $\hat{\pi}$ an automorphism of $Y$ extending $\pi$. Denote the minimum elements of $A$ and $A'$ as $\alpha^*$ and $\beta^*$, respectively. Then $\alpha^* \pi = \beta^*$, and for each $\delta\in Z$ the diagram  
\begin{align} \label{3} \xymatrix{
A_{\delta} \ar[d]^{{\psi}^{A}_{\delta, \alpha^*}} \ar[r]^{\theta_{\delta}} &A_{\delta \pi}' \ar[d]^{{\psi}^{A'}_{\delta \pi, \beta^*}} \\\
A_{\alpha^*} \ar[r]^{\theta_{\alpha^*}} &A_{\beta^*}'}
\end{align}  
commutes by Proposition \ref{iso normal}. 
By Corollary \ref{rec bands homog} we may extend $\theta_{\alpha^*}$ to an isomorphism $\hat{\theta}_{\alpha^*}:B_{\alpha^*} \rightarrow B_{\beta^*}$. 
For each $\delta\in Y$, let $\hat{\theta}_{\delta}:B_{\delta}\rightarrow B_{\delta \hat{\pi}}$ be the isomorphism given by 
\[ \hat{\theta}_{\delta}= {\psi}_{\delta,\alpha^*} \, \hat{\theta}_{\alpha^*} \, {\psi}_{\alpha^*\pi,\delta \hat{\pi}}
\] 
so that $\hat{\theta}=[\hat{\theta}_{\delta},\hat{\pi}]_{\delta\in Y}$ is an automorphism of $B$ by Lemma \ref{iso surj}.
 Then $\hat{\theta}_{\delta}$ extends $\theta_{\delta}$ for each $\delta\in Z$, since by \eqref{3},
\[ \theta_{\delta}= \psi_{\delta,\alpha^*}^{A} \, \theta_{\alpha^*} \, (\psi^{A'}_{\delta\pi,\beta^*})^{-1}|_{\text{Im } \psi^{A'}_{\delta\pi,\beta^*}},
\]  
where $\psi_{\delta,\alpha^*}$ extends $\psi_{\delta,\alpha^*}^A$,  $\hat{\theta}_{\alpha^*}$ extends $\theta_{\alpha^*}$ and ${\psi}_{\beta^*,\delta\hat{\pi}}$ extends $(\psi^{A'}_{\delta\pi,\beta^*})^{-1}|_{\text{Im } \psi^{A'}_{\delta\pi,\beta^*}}$. Hence $\hat{\theta}$ extends $\theta$, and $B$ is structure-homogeneous. 
 \end{proof} 

Moreover, it follows from the proceeding lemma that for any semilattice $Y$ and $n,m\in \mathbb{N}^*$, there exists a unique, up to isomorphism, normal band with connecting morphisms being isomorphisms, structure semilattice isomorphic to $Y$ and $\mathcal{D}$-classes isomorphic to $B_{n,m}$. The result can be obtained from \cite{Petrich74} but is proven here for completeness. 

\begin{lemma}\label{iso state} Let $B=[Y;B_{\alpha};\psi_{\alpha,\beta}]$ be a normal band such that each connecting morphism is an isomorphism. Then $B \cong Y \times B_{n,m}$ for some $n,m\in \mathbb{N}^*$. Conversely, every band $Y \times B_{n,m}$ is isomorphic to some normal band such that each connecting morphism is an isomorphism. 
\end{lemma}

\begin{proof} Since each connecting morphism is an isomorphism, it follows that the $\mathcal{D}$-classes are pairwise isomorphic. Suppose $B_{\alpha}\cong B_{n,m}$ for each $\alpha\in Y$. Fix $\delta\in Y$ and let $\phi_{\delta}:B_{\delta} \rightarrow B_{n,m}$ be an isomorphism. Then the map $\theta:B\rightarrow Y \times B_{n,m}$ given by $e_{\alpha}\theta=(\alpha,e_{\alpha}\psi_{\alpha,\delta}\phi_{\delta})$ is a bijection. Moreover, if $ e_{\alpha},e_{\gamma}\in B$ then 
\begin{align*}
e_{\alpha}\theta \, e_{\gamma}\theta & = (\alpha, e_{\alpha}\psi_{\alpha,\delta}\phi_{\delta})(\gamma,e_{\gamma}\psi_{\gamma,\delta}\phi_{\delta}) \\
& = (\alpha\gamma, [(e_{\alpha}\psi_{\alpha,\alpha\gamma}\psi_{\alpha\gamma,\delta})(e_{\gamma}\psi_{\gamma,\alpha\gamma}\psi_{\alpha\gamma,\delta})]\phi_{\delta}) \\
& = (\alpha\gamma, ( e_{\alpha}\psi_{\alpha,\alpha\gamma} \, e_{\gamma} \psi_{\gamma,\alpha\gamma})\psi_{\alpha\gamma,\delta}\phi_{\delta}) \\
& = (e_{\alpha}e_{\gamma})\theta
\end{align*} 
and so $\theta$ is an isomorphism. 

Conversely, let $T_{\alpha}=\{(\alpha,e):e\in B_{n,m}\}$ for each $\alpha\in Y$. Clearly $T_{\alpha}$ is isomorphic to $B_{n,m}$. For each $\alpha\geq \beta$ in $Y$, let $\varphi_{\alpha,\beta}:T_{\alpha}\rightarrow T_{\beta}$ be the isomorphism given by $(\alpha,e)\varphi_{\alpha,\beta}=(\beta,e)$. Then $[Y;T_{\alpha};\varphi_{\alpha,\beta}]$ forms a strong semilattice of rectangular bands, and is isomorphic to $Y\times B_{n,m}$ by the forwards direction to the proof. 
\end{proof} 

Let $R$ be a right normal band with homogeneous structure semilattice $Y$. Then as $Y \times B_{n,1}$ is structure-homogeneous for any $n\in \mathbb{N}^*$, it follows from Corollary \ref{SH regular} that we may let $(Y \times B_{n,1}) \bowtie R$ denote the unique, up to isomorphism, normal band with left component isomorphic to $(Y \times B_{n,1})$ and right component isomorphic to $R$. Note that $Y \times B_{n,m}\cong (Y \times B_{n,1})\bowtie (Y \times B_{1,m})$. 

Furthermore, for any $n\in \mathbb{N}^*$ and homogeneous bands $R$ and $L$, where $R$ is right normal and $L$ is left normal, the bands $(Y \times B_{n,1}) \bowtie R$ and $L \bowtie (Y \times B_{1,n})$ are homogeneous by Corollary \ref{structure homog reg} and Proposition \ref{iso hom}.  

Finally, we examine the case where the connecting morphisms are surjective but not injective (so that the $\mathcal{D}$-classes are infinite). Let $B=[Y;B_{\alpha};\psi_{\alpha,\beta}]$ be a surjective normal band. For each $\alpha > \beta$, let $K_{\alpha,\beta}$ denote the congruence 
\[ \text{Ker }\psi_{\alpha,\beta}=\{(e_{\alpha},f_{\alpha}):e_{\alpha}\psi_{\alpha,\beta}=f_{\alpha}\psi_{\alpha,\beta}\}
\] 
 on $B_{\alpha}$,  or $K_{\alpha,\beta}^B$ if we need to distinguish the band $B$. Note that if $\alpha > \beta > \gamma$ then $K_{\alpha,\beta} \subseteq K_{\alpha,\gamma}$, for if $e_{\alpha}\psi_{\alpha,\beta}=f_{\alpha}\psi_{\alpha,\beta}$ then 
$e_{\alpha}\psi_{\alpha,\gamma}=e_{\alpha}\psi_{\alpha,\beta}\psi_{\beta,\gamma}=f_{\alpha}\psi_{\alpha,\beta}\psi_{\beta,\gamma}=f_{\alpha}\psi_{\alpha,\gamma}.
$ We then obtain what can be considered the dual of Lemma \ref{trivial not universal}: 

\begin{lemma}\label{not iso uni} Let $B=L \bowtie R$ be a homogeneous normal band such that the connecting morphisms of either $L$ or $R$ are surjective but not injective.
 Then $Y$ is the universal semilattice. 
\end{lemma} 

\begin{proof} Suppose w.l.o.g. that $R=[Y;R_{\alpha};\psi_{\alpha,\beta}^r]$ has surjective but not injective connecting morphisms, so $|R_{\alpha}|=\aleph_0$ for all $\alpha\in Y$. Suppose for contradiction that $Y$ is a linear or semilinear order. Let $e_{\alpha}, f_{\alpha},g_{\alpha}\in R_{\alpha}$ be such that $(e_{\alpha},f_{\alpha})\in K^R_{\alpha,\beta}$ but $(e_{\alpha},g_{\alpha}) \notin K^R_{\alpha,\beta}$, noting that such elements exist as $\psi^r_{\alpha,\beta}$ is surjective but not injective. For any $l_{\alpha}\in L_{\alpha}$, extend the automorphism of the right zero subband $\{(l_\alpha,e_\alpha),(l_\alpha,f_\alpha),(l_\alpha,g_\alpha)\}$  which fixes $(l_{\alpha},e_{\alpha})$ and swaps $(l_{\alpha},f_{\alpha})$ and $(l_{\alpha},g_{\alpha})$ to an automorphism $\theta=[\theta_{\alpha},\pi]_{\alpha\in Y}$ of $B$. Then $\theta=\theta^l\bowtie \theta^r$ for some automorphisms $\theta^l=[\theta_{\alpha}^l,\pi]_{\alpha\in Y}$ and $\theta^r=[\theta_{\alpha}^r,\pi]_{\alpha\in Y}$ of $L$ and $R$, respectively. It follows by the commutativity of the diagram $[\alpha,\beta;\alpha,\beta\pi]$ in $R$ that 
\[ (e_{\alpha},g_{\alpha})\in K_{\alpha,\beta\pi} \, \text{ and } \, (e_{\alpha},f_{\alpha})\notin K_{\alpha,\beta\pi}.
\] However as $\{\gamma:\gamma<\alpha\}$ is a chain, either $\beta<\beta\pi$ or $\beta>\beta\pi$, which both contradict the note above the lemma. Hence $Y$ contains a diamond and, being homogeneous by Corollary \ref{Y homog surj}, is thus the universal semilattice. 
\end{proof}

To complete the classification of homogeneous surjective normal bands, we use a general method of Fra\"iss\'e for obtaining homogeneous structures. Here we apply this only to semigroups. 
Let $\mathcal{K}$ be a class of f.g.  semigroups. Then we say   
\begin{enumerate} [label=(\arabic*)] 
\item $\mathcal{K}$ is \textit{countable} if  it contains only countably many isomorphism types.
\item  $\mathcal{K} $ is \textit{closed under isomorphism} if whenever $A\in \mathcal{K}$ and $B\cong A$ then $B\in \mathcal{K}$.  
\item $\mathcal{K}$ has the \textit{hereditary property} (HP) if given $A\in \mathcal{K}$ and $B$ a f.g.  subsemigroup of $A$ then $B\in \mathcal{K}$. 
\item  $\mathcal{K}$ has the \textit{joint embedding property} (JEP) if given $B_1,B_2\in \mathcal{K}$, then there exists $C\in \mathcal{K}$ and embeddings $f_i:B_i\rightarrow C$ ($i=1,2$).
\item  $\mathcal{K}$ has the \textit{amalgamation property}\footnote{This is also known as the \textit{weak amalgamation property}.}  (AP) if given  $A, B_1, B_2\in \mathcal{K}$, where $A$ is non-empty, and embeddings $f_i:A\rightarrow B_i$ ($i=1,2$), then there exists $D\in \mathcal{K}$ and embeddings $g_i: B_i \rightarrow D$  such that 
\[     f_1 \circ g_1 = f_2 \circ g_2. 
\] 
\end{enumerate} 
 
The \textit{age} of a semigroup $S$ is the class of all f.g. semigroups which can be embedded in $S$. 

 \begin{theorem}[Fra\"iss\'e's Theorem for semigroups]\cite{Fraisse} Let  $\mathcal{K}$ be a non-empty countable class of f.g.  semigroups which  is closed under isomorphism and satisfies HP, JEP, and AP. Then there exists a unique, up to isomorphism, countable homogeneous semigroup $S$ such that $\mathcal{K}$ is the age of $S$. Conversely, the age of a countable homogeneous semigroup is closed under isomorphism, is countable and satisfies HP, JEP and AP. 
\end{theorem}

We call $S$ the \textit{Fra\"iss\'e limit} of $\mathcal{K}$. 

Let $\mathcal{K}$ be a  Fra\"iss\'e subclass of a variety of bands $\mathcal{V}$ defined by the identity $a_1a_2\cdots a_n=b_1b_2\cdots b_m$. Then the Fra\"iss\'e limit $S$ of $\mathcal{K}$ is a member of $\mathcal{V}$. Indeed, if $x_1,x_2,\dots,x_n, y_1,y_2,$ 
$ \dots ,y_m\in S$ then $\langle x_1,x_2,\dots,x_n,y_1,y_2,\dots ,y_m \rangle \in \mathcal{K}$ and thus $x_1x_2 \cdots x_n=y_1y_2\cdots y_m$.

\begin{example} The rectangular band  $B_{\aleph_0,\aleph_0}$ is homogeneous by Corollary \ref{rec bands homog}, and clearly its age is the class of all finite rectangular bands. It follows that the class of all finite rectangular bands forms a Fra\"iss\'e class (with Fra\"iss\'e limit $B_{\aleph_0,\aleph_0}$). 
\end{example}
\begin{example} The class of all finite semilattices forms a Fra\"iss\'e class, with Fra\"iss\'e limit the universal semilattice (see \cite{Hall75}, for example). \end{example}

\begin{example} Let $\mathcal{K}$ be the class of all finite bands. Since the class of all bands forms a variety, $\mathcal{K}$ satisfies HP and is closed under (finite) direct product, and thus has JEP. However, it was shown in \cite[Page 12]{Imaoka76} that AP does not hold. 
\end{example} 

Consequently, there does not exist a universal homogeneous band, that is, one which embeds every finite band.
 However, if we refine our class to certain normal bands, AP is shown to hold. To this end, let $\mathcal{K}_N, \mathcal{K}_{RN}$ and $\mathcal{K}_{LN}$ be the classes of finite normal, finite right normal and finite left normal bands, respectively. 

\begin{lemma}\label{universal normals} The classes $\mathcal{K}_N, \mathcal{K}_{RN}$ and  $\mathcal{K}_{LN}$ form Fra\"iss\'e classes. 
\end{lemma} 

 \begin{proof}  Since the class of (left/right) normal bands forms a variety, it is clear that the classes are closed under subbands and have JEP. The weak amalgamation property follows from \cite[Section 2]{Imaoka76} by taking all bands to be finite. Finally, since bands are ULF there exists only finitely many bands, up to isomorphism, of each finite cardinality, and so each class is countable.  
\end{proof} 

Let  $\mathcal{B}_N, \mathcal{B}_{RN}$ and $\mathcal{B}_{LN}$ be the Fra\"iss\'e limits of $\mathcal{K}_N, \mathcal{K}_{RN}$ and  $\mathcal{K}_{LN}$, respectively. We shall prove that $\mathcal{B}_{RN}$ is the unique homogeneous right normal band with surjective but not injective connecting morphisms. This will follow quickly from the subsequent result. 

\begin{lemma} Let $R=[Y;R_{\alpha};\psi_{\alpha,\beta}]$ be a homogeneous right normal band, where each connecting morphism is surjective but not injective. Let $\alpha>\beta_1,\dots,\beta_r$ in $Y$ for some $r\in \mathbb{N}$, where $\beta_i \perp \beta_j$ for all $i\neq j$. Then for any $e_{\beta_i}\in B_{\beta_i}$ such that $\langle e_{\beta_i}:1 \leq i \leq r\rangle$ forms a semilattice, we have 
\[ |\{e_{\alpha}\in R_{\alpha}:e_{\alpha}\psi_{\alpha,\beta_i}=e_{\beta_i} \text{ for all }  1\leq i \leq r\}|=\aleph_0.
\]  
\end{lemma} 

\begin{proof} By Lemma \ref{not iso uni}, $Y$ is the universal semilattice, and so every pair of elements has an upper cover. We first prove the result for $r=1$ (relabelling $\beta_1$ simply as $\beta$). Since the connecting morphisms are surjective, there exists $e_{\alpha}\in B_{\alpha}$ such that $e_{\alpha}\psi_{\alpha,\beta}=e_{\beta}$. Suppose for contradiction that 
\[ e_{\alpha}K_{\alpha,\beta}=\{f_{\alpha}:(e_{\alpha},f_{\alpha})\in K_{\alpha,\beta}\}
\] 
has finite cardinality $n$. Note that $n\neq 1$ since the connecting morphisms are not injective and $|e_{\alpha}K_{\alpha,\beta}|=|e_{\alpha'}K_{\alpha',\beta'}|$ for all $\alpha'>\beta'$ and $e_{\alpha'}\in R_{\alpha'}$, by a simple application of homogeneity. Moreover, for any $\gamma<\beta$ we have $|e_{\alpha} K_{\alpha,\beta}|=|e_{\alpha} K_{\alpha,\gamma}|$ and $K_{\alpha,\beta}\subseteq K_{\alpha,\gamma}$, and so $e_{\alpha}K_{\alpha,\beta}=e_{\alpha}K_{\alpha,\gamma}$. Let $e_{\beta}\psi_{\beta,\gamma}=e_{\gamma}$. Then choosing any $f_{\beta}\in e_{\beta}K_{\beta,\gamma}$ with $f_{\beta}\neq e_{\beta}$ there exists $f_{\alpha}\in R_{\alpha}$ such that $f_{\alpha}\psi_{\alpha,\beta}=f_{\beta}$, and thus 
\[f_{\alpha}\psi_{\alpha,\gamma}=f_{\beta}\psi_{\beta,\gamma}=e_{\beta}\psi_{\beta,\gamma}=e_{\gamma}.
\] 
Hence $f_{\alpha}\in e_{\alpha}K_{\alpha,\gamma}$, but $f_{\alpha}\not\in e_{\alpha}K_{\alpha,\beta}$, a contradiction and thus $n$ is infinite. 

Now consider the result for arbitrary $r\in \mathbb{N}$. Let $f_{\alpha}\in R_{\alpha}$, and let $f_{\alpha}\psi_{\alpha,\beta_i}=f_{\beta_i}$ for some $f_{\beta_i}$. Note that $\langle f_{\beta_i}:1\leq i \leq r \rangle$ is a semilattice, and is isomorphic to $\langle e_{\beta_i}:1\leq i \leq r \rangle$. We may thus extend the isomorphism between $\langle f_{\beta_i}:1\leq i \leq r \rangle$  and $\langle e_{\beta_i}:1\leq i \leq r \rangle$ which sends  $f_{\beta_i}$ to $e_{\beta_i}$ for each $i$, to an automorphism of $B$, to obtain some $\delta>\beta_i$ and $e_{\delta}\in B_{\delta}$ (as the image of $e_{\alpha}$) such that $e_{\delta}\psi_{\delta,\beta_i}=e_{\beta_i}$ for each $i$. Since $Y$ is the universal semilattice we may pick $\tau>\alpha,\delta$ in $Y$. Let  $e_{\tau}$ be such that $e_{\tau}\psi_{\tau,\delta}=e_{\delta}$, and suppose $e_{\tau}\psi_{\tau,\alpha}=e_{\alpha}$. Then 
\[ e_{\alpha}\psi_{\alpha,\beta_i}=e_{\tau}\psi_{\tau,\alpha}\psi_{\alpha,\beta_i}= e_{\tau}\psi_{\tau,\beta_i}=e_{\tau}\psi_{\tau,\delta}\psi_{\delta,\beta_i}=e_{\delta}\psi_{\delta,\beta_i}=e_{\beta_i}.
\] 
By the case where $r=1$ the set $e_{\tau}K_{\tau,\delta}$ is infinite, and thus so is the set  
\[ \{e_{\tau}\in R_{\tau}:e_{\tau}\psi_{\tau,\beta_i}=e_{\beta_i} \text{ for all }  1\leq i \leq r\}.
\] 
The result follows by extending the isomorphism from the semilattice $\langle e_{\tau},e_{\beta_i}:1\leq i \leq r \rangle$ to $\langle e_{\alpha},e_{\beta_i}: 1 \leq i \leq r \rangle$, which sends $e_{\tau}$ to $e_{\alpha}$ and fixes all other elements, to an automorphism of $R$. 

\end{proof} 

\begin{lemma}\label{universal right} Let $R=[Y;R_{\alpha};\psi_{\alpha,\beta}]$ be a homogeneous right normal band, where each connecting morphism is surjective but not injective. Then $R\cong \mathcal{B}_{RN}$ (dually for $\mathcal{B}_{LN}$).
\end{lemma}

\begin{proof} We shall prove that all finite right normal bands embed in $R$. We proceed by induction, the base case being trivially true, by supposing that all right normal bands of size $n-1$ embed in $R$, and let $A=[Z;A_{\alpha};\phi_{\alpha,\beta}]$ be of size $n$. Let $\alpha$ be maximal in $Z$ and fix $e_{\alpha}\in A_{\alpha}$. Suppose $\alpha$ is the cover of $\beta_1,\dots,\beta_r$ in $Z$, and suppose $e_{\alpha}\phi_{\alpha,\beta_i}=e_{\beta_i}$. Then $A'=A\setminus \{e_{\alpha}\}=[\bar{Z};A'_{\alpha};\phi_{\alpha,\beta}']$ is a right normal band of size $n-1$, and so there exists an embedding $\theta: A' \rightarrow R$ (which induces an embedding $\pi:\bar{Z} \rightarrow Y$). 
Since $Y$ is the universal semilattice by Lemma \ref{not iso uni}, and in particular embeds all finite semilattices, it follows that  there exists $\delta\in Y$ such that $\bar{Z}\pi \cup \{ \delta\} \cong Z$, where we choose $\delta=\alpha\pi$ if $|A_{\alpha}|>1$, that is, if $\bar{Z}=Z$. Then by the previous lemma, we may pick an element $e_{\delta}$ of $R_{\delta}$ such that $e_{\delta} \not \in A'\theta$ and $e_{\delta}\psi_{\delta,\beta_i\pi}=e_{\beta_i}\theta$. Then it is easily verifiable that $A'\theta \cup \{e_{\delta}\}$ is isomorphic to $A$, and so the result follows by induction. By Fra\"iss\'e's Theorem $R$ is isomorphic to the Fra\"iss\'e limit of $\mathcal{K}_{RN}$. 
\end{proof} 

\begin{corollary}\label{BN=BLN RN} The band $\mathcal{B}_{N}$ is isomorphic to $\mathcal{B}_{LN}\bowtie \mathcal{B}_{RN}$. 
\end{corollary} 
\begin{proof} Let $L\bowtie R$ be a finite normal band with structure semilattice $Z$. Then there exists embeddings $\theta^l:L\rightarrow \mathcal{B}_{LN}$ and $\theta^r:R\rightarrow \mathcal{B}_{RN}$ with induced embeddings $\pi_l$ and $\pi_r$ from $Z$ to $Y$, respectively. Hence $\pi=(\pi_l)^{-1}|_{Z\pi_l}\pi_r$ is an isomorphism between $Z\pi_l$ to $Z\pi_r$.
 By Lemma \ref{embed Y} there exists subbands $A=\{e_{\alpha}:\alpha\in Z\pi_l\}$ and $A'=\{f_{\alpha}:\alpha\in Z\pi_r\}$ of $\mathcal{B}_{LN}$ isomorphic to $Z\pi_l$ and $Z\pi_r$, respectively. Consequently, the map $\phi:A\rightarrow A'$ given by $e_{\alpha}\phi = f_{\alpha\pi}$ ($\alpha\in Z\pi_l$) is an isomorphism, which we may extend to an automorphism $\hat{\theta}^l=[\hat{\theta}_{\alpha}^l,\hat{\pi}]$ of $\mathcal{B}_{LN}$. 
 In particular, $\hat{\pi}$ extends $\pi$ and  $\theta^l\hat{\theta}^l$ is an embedding of $L$ into $\mathcal{B}_{LN}$, with induced embedding $\pi_l\hat{\pi}=\pi_l(\pi_l)^{-1}|_{Z\pi_l}\pi_r=\pi_r$ of $Z$ into $Y$.
  Hence $\theta^l\hat{\theta}^l\bowtie \theta^r: L\bowtie R\rightarrow \mathcal{B}_{LN}\times \mathcal{B}_{RN}$ is an embedding by Proposition \ref{iso regular}, and so $\mathcal{B}_{LN}\bowtie \mathcal{B}_{RN}$ embeds all finite normal bands as required. 
\end{proof} 

We now summarise our findings in this section.  

\begin{proposition}\label{homog surj} A surjective normal band is homogeneous if and only if it is isomorphic to either $Y\times B_{n,m}$, $(U\times B_{n,1})\bowtie \mathcal{B}_{RN}, \mathcal{B}_{LN}\bowtie (U \times B_{1,n})$ or $\mathcal{B}_N$, for some homogeneous semilattice $Y$ and some $n,m\in \mathbb{N}^*$, where $U$ is the universal semilattice. 
\end{proposition} 

\begin{proof} 
Suppose first that $B=L\bowtie R$ is a homogeneous surjective normal band. Then by Corollary \ref{Y homog surj} and Corollary \ref{surj split}, each of $Y, L$ and $R$ are homogeneous.
If a non-trivial connecting morphism of $L$ is an isomorphism, then $L$  is isomorphic to  $Y \times B_{n,1}$ by Lemma \ref{image same} and Lemma \ref{iso state}.  Otherwise, the connecting morphisms of $L$ are non-injective and $L$ is isomorphic to $\mathcal{B}_{LN}$ by  Lemma \ref{universal right}. Dually for $R$. 
  Since the band $Y\times B_{n,m}$ is structure homogeneous for any homogeneous semilattice $Y$ by Proposition \ref{iso hom}, the result  follows by Corollary \ref{SH regular},  Lemma \ref{not iso uni} and Corollary \ref{BN=BLN RN}.

Conversely, $\mathcal{B}_N,\mathcal{B}_{RN},\mathcal{B}_{LN}$ are homogeneous by Fra\"iss\'e's Theorem. 
Since each $Y \times B_{n,m}$ is structure-homogeneous, the final cases are homogeneous by Corollary \ref{structure homog reg}.
\end{proof}

For a complete classification of homogeneous normal bands, it thus suffices to consider the spined product of an image-trivial normal band with a surjective normal band. 

To this end, let $B=L\bowtie R$ be a homogeneous normal band, where $L$ is image-trivial and $R$ is surjective. We may also assume $L$ and $R$ are not semilattices, since otherwise $B$ would be image-trivial or surjective. Then $L$ is homogeneous by Corollary \ref{trivial splits}, and so $L\cong T_{n,1,k}$ for some $n,k\in \mathbb{N}^*$. Since the structure semilattice of $B$ is a semilinear order, it follows from Lemma \ref{not iso uni} that the connecting morphisms of $R$ must be isomorphisms, and thus we may assume that $R= T_{nk} \times B_{1,m}$ for some $m\in \mathbb{N}^*$. Conversely, $T_{n,1,k}\bowtie ( T_{nk} \times B_{1,m})$ is homogeneous for any $n,m,k\in \mathbb{N}^*$ by Proposition \ref{iso hom} and Corollary \ref{structure homog reg}.

This, together with its dual and Proposition \ref{T homog} and Proposition \ref{homog surj}, gives a complete list of homogeneous normal bands. In the classification theorem below, the three cases (up to duality) are given by: image-trivial normal bands in $(\mathrm{i})$, surjective normal bands in $(\mathrm{ii}),(\mathrm{iii}), (\mathrm{iv})$,$(\mathrm{v})$, and finally the spined product of an image-trivial normal band with a surjective normal band in $(\mathrm{vi})$ and $(\mathrm{vii})$. 

\begin{theorem}[Classification Theorem of homogeneous normal bands] A normal band is homogeneous if and only if it is isomorphic to either: 
\begin{enumerate}[label=(\roman*), font=\normalfont]
\item  $T_{n,m,k}$;
\item $Y \times B_{n,m}$; 
\item $\mathcal{B}_{LN} \bowtie (U \times B_{1,m})$; 
\item $(U \times B_{n,1} ) \bowtie \mathcal{B}_{RN}$; 
\item $\mathcal{B}_N$; 
\item $(T_{mk} \times B_{n,1}) \bowtie T_{1,m,k}$; 
\item $T_{n,1,k} \bowtie (T_{nk} \times B_{1,m} )$; 
\end{enumerate}  
for some homogeneous semilattice $Y$ and some $n,m,k\in \mathbb{N}^*$, where $U$ is the universal semilattice. 
\end{theorem}

We finish this section by giving a complete classification of structure-homogeneous normal bands. 

\begin{proposition}\label{SH normal} A normal band is structure-homogeneous if and only if isomorphic to $Y \times B_{n,m}$ for some homogeneous semilattice $Y$ and $n,m\in \mathbb{N}^*$. 
\end{proposition} 

\begin{proof}
Let $B=[Y;B_{\alpha};\psi_{\alpha,\beta}]$ be a structure-homogeneous, so that $Y$ is homogeneous by Corollary \ref{Y homog surj}.
 We shall show that each connecting morphism is an isomorphism, so that the result will follow by Lemma \ref{iso state}. Suppose first that $\psi_{\alpha,\beta}$ is not surjective, say, $e_{\beta}\not\in I_{\alpha,\beta}$. 
 Let $f_{\alpha}\in B_{\alpha}$ with $f_{\alpha}\psi_{\alpha,\beta}=f_{\beta}$. Then by extending the isomorphism between $\{e_{\beta}\}$ and $\{f_{\beta}\}$ (with induced isomorphism the trivial map fixing $\beta$) to an automorphism of $B$ with induced automorphism $1_Y$, a contradiction is obtained. Hence each connecting morphism is surjective. 

Suppose $e_{\alpha}\psi_{\alpha,\beta}=e_{\beta}=f_{\alpha}\psi_{\alpha,\beta}$, and fix some $\delta>\alpha$. Then as $\psi_{\delta,\alpha}$ is surjective there exists $e_{\delta},f_{\delta}\in B_{\delta}$ with $e_{\delta}\psi_{\delta,\alpha}=e_{\alpha}$ and $f_{\delta}\psi_{\delta,\alpha}=f_{\alpha}$. Let $\pi$ be an automorphisms of $Y$ such that $\alpha\pi=\beta$ and $\delta\pi=\delta$ (such a map exists by the homogeneity of $Y$). Extend the isomorphism swapping $\{e_{\delta}\}$ and $\{f_{\delta}\}$ to an automorphism $\theta=[\theta_{\alpha},\pi]_{\alpha\in Y}$ of $B$. Then as the diagram $[\delta,\alpha;\delta,\beta]$ commutes,  
\[ e_{\alpha}\theta_{\alpha}=e_{\delta}\psi_{\delta,\alpha}\theta_{\alpha}=e_{\delta}\theta_{\delta}\psi_{\delta,\beta} = f_{\delta}\psi_{\delta,\beta}=e_{\beta} 
\] 
and similarly $f_{\alpha}\theta_{\alpha}=e_{\beta}$. Hence $e_{\alpha}=f_{\alpha}$, and so $\psi_{\alpha,\beta}$ is injective. 

The converse follows from Proposition \ref{iso hom}. 
\end{proof}

\section{Homogeneous linearly ordered bands} \label{sec:linear}

We call a band $B=\bigcup_{\alpha\in Y}B_{\alpha}$ \textit{linearly ordered} if  $Y$ is a linear order. A homogeneous linearly ordered band $B$ has structure semilattice $\mathbb{Q}$ by Proposition \ref{linear homog}. 
We  observe that if $B$ is not normal, then there exists $\alpha>\beta$ and $e_{\alpha}\in B_{\alpha}$ such that the subband $B_{\beta}(e_{\alpha})$ of $B_{\beta}$ contains more than one $\mathcal{L}$-class or $\mathcal{R}$-class. Hence if $B$ is homogeneous, then by Corollary \ref{basics} ($\mathrm{iv}$) the same is true for $B_{\gamma}(e_{\delta})$ for any  $\delta>\gamma$ and $e_{\delta}\in B_{\delta}$.

\begin{lemma}\label{all R} Let $B$ be a linearly ordered homogeneous non-normal band. If $B_{\beta}(e_{\alpha})$ intersects more than one $\mathcal{L}$-class, then $B_{\beta}(e_{\alpha})= \mathcal{R}(B_{\beta}(e_{\alpha}))$ (dually for $\mathcal{R}$). 
\end{lemma} 

\begin{proof} Suppose for contradiction that there exists $g_{\beta}\in B_{\beta}$ such that $g_{\beta}<_{r} e_{\alpha}$ but $g_{\beta} \not< e_{\alpha}$. Then $g_{\beta}=e_{\alpha}g_{\beta}$ and so $g_{\beta}e_{\alpha}=e_{\alpha}g_{\beta}e_{\alpha}  \, \mathcal{R} \, g_{\beta}$ and $g_{\beta}e_{\alpha}\in B_{\beta}(e_{\alpha})$.  Since the subband $B_{\beta}(e_{\alpha})$ contains more than one $\mathcal{L}$-class, we may pick $f_{\beta} \, \mathcal{R} \, {g_{\beta}e_{\alpha}}$ such that $f_{\beta}\in B_{\beta}(e_{\alpha})\setminus\{g_{\beta}e_{\alpha}\}$. Extend the automorphism of the right zero subsemigroup $\{f_\beta, g_{\beta}, g_{\beta}e_{\alpha}\}$ which fixes $g_{\beta}e_{\alpha}$ and swaps $f_{\beta}$ and $g_{\beta}$ to $\theta\in \text{Aut}(B)$. Then $e_{\alpha'}=e_{\alpha}\theta>g_{\beta}e_{\alpha},g_{\beta}$ and $g_{\beta}e_{\alpha}=(g_{\beta}e_{\alpha})\theta=g_{\beta}\theta e_{\alpha}\theta=f_{\beta}e_{\alpha'}$. Hence 
\[ f_{\beta}(e_{\alpha}e_{\alpha'}e_{\alpha})=f_{\beta}e_{\alpha'}e_{\alpha}=g_{\beta}e_{\alpha}e_{\alpha}=g_{\beta}e_{\alpha}
\] 
and so $f_{\beta}\not\leq  e_{\alpha}e_{\alpha'}e_{\alpha}$. If $\alpha'\geq \alpha$ then $e_{\alpha}e_{\alpha'}e_{\alpha}=e_{\alpha}$, contradicting $e_{\alpha}>f_{\beta}$. Hence $\alpha'<\alpha$, so that $e_{\alpha'}e_{\alpha}e_{\alpha'}=e_{\alpha'}$ and 
\[ g_{\beta}=g_{\beta}e_{\alpha'}=g_{\beta}e_{\alpha'}e_{\alpha}e_{\alpha'}=g_{\beta}e_{\alpha}e_{\alpha'}=g_{\beta}e_{\alpha},
\] 
a contradiction.
\end{proof}

Let $B$ be a linearly ordered homogeneous non-normal band. Then by the lemma above, if $B_{\beta}(e_{\alpha})$ contains a square (that is, it intersects more than one $\mathcal{L}$ and $\mathcal{R}$-class) then $B_{\beta}(e_{\alpha})=B_{\beta}$. Thus $B_{\beta}(e_{\alpha})$ is a single $\mathcal{K}$-class, where $\mathcal{K}=\mathcal{L,R}$ or $\mathcal{D}$. Moreover, for all $\alpha'>\beta'$ and $e_{\alpha'}\in B_{\alpha'}$, it follows from Corollary \ref{basics} $(\mathrm{iv})$ that every $B_{\beta'}(e_{\alpha'})$ is also a single $\mathcal{K}$-class. 

Note that if $\mathcal{K}=\mathcal{D}$, then $B$ has the following property, which we shall follow \cite{Petrich74} in calling \textit{$\mathcal{D}$-covering}: 
\[ \text{If } e,f\in B, \text{ then either } e \, \mathcal{D} \, f \text{ or } e>f \text{ or } e<f.
\] 

\begin{proposition}\label{lin band reg} A homogeneous linearly ordered band is regular. 
\end{proposition} 

\begin{proof} Let $B=\bigcup_{\alpha\in \mathbb{Q}}B_{\alpha}$ be a homogeneous non-normal linearly ordered band (noting that if $B$ was normal, then it would automatically be regular). By Lemma \ref{all R} we may assume first that each $B_{\beta}(e_{\alpha})$ is either a single $\mathcal{R}$-class or is $B_{\beta}$ (noting in both cases, $B_{\beta}(e_{\alpha})$ is a union of $\mathcal{R}$-classes). Hence  $\mathcal{L}(B_{\beta}(e_{\alpha}))=B_{\beta}$, and so $f_{\beta}e_{\alpha}=f_{\beta}$ for all $f_{\beta}\in B_{\beta}$. 
Given any $\gamma,\tau,\sigma\in Y$ and any elements $e_{\gamma}\in B_\gamma ,f_{\tau}\in B_\tau$ and $g_{\sigma}\in B_\sigma$, it suffices to show that
\begin{equation}\label{regular eq} e_{\gamma}f_{\tau}e_{\gamma}g_{\sigma}e_{\gamma}=e_{\gamma}f_{\tau}g_{\sigma}e_{\gamma}.
\end{equation}
If $\tau<\gamma$ then $f_{\tau}e_{\gamma}=f_{\tau}$, while if $\gamma<\tau$ then $e_{\gamma}f_{\tau}=e_{\gamma}$, and \eqref{regular eq} is seen to hold in both cases. Assume instead that $\tau = \gamma$ and $\gamma>\sigma$ (since if $\gamma\leq \sigma$ then both sides of \eqref{regular eq} cancel to $e_{\gamma}$). Then $e_{\gamma}g_{\sigma} \, \mathcal{L} \, g_{\sigma} \, \mathcal{L} \, f_{\gamma}g_{\sigma} \, \mathcal{L} \, e_{\gamma}f_{\gamma}g_{\sigma}$ and
\[ (e_{\gamma}g_{\sigma})(e_{\gamma}f_{\gamma}g_{\sigma})= e_{\gamma}g_{\sigma} f_{\gamma}g_{\sigma} = e_{\gamma}g_{\sigma}
\] 
so that $e_{\gamma}g_{\sigma} \, \mathcal{R} \, e_{\gamma}f_{\gamma}g_{\sigma}$, and thus $e_{\gamma}g_{\sigma} = e_{\gamma}f_{\gamma}g_{\sigma}$. By post-multiplying by $e_{\gamma}$, and noting that $e_{\gamma}=e_{\gamma}f_{\tau}e_{\gamma}$ we obtain \eqref{regular eq}, and thus $B$ is regular. The case where each $B_{\beta}(e_{\alpha})$ is a union of $\mathcal{L}$-classes is proven dually. 
\end{proof}

Let $B=L\bowtie R$ be a homogeneous non-normal linearly ordered band, where $L=\bigcup_{\alpha\in \mathbb{Q}} L_{\alpha}$ and $R=\bigcup_{\alpha\in \mathbb{Q}}R_{\alpha}$. Then for any finite chain $\alpha_1>\alpha_2>\cdots >\alpha_n$ in $\mathbb{Q}$, we pick  $l_{\alpha_1}\in L_{\alpha_1}$ to construct a chain $l_{\alpha_1}>l_{\alpha_2}>\cdots >l_{\alpha_n}$ in $L$.
By an identical argument to the proof of Corollary \ref{surj split}, we have that $R$ is homogeneous, and dually so is $L$. Hence by Lemma \ref{all R}, each $L_{\beta}(l_{\alpha})$ is a single $\mathcal{R}$ or $\mathcal{L}$-class of $L$. Since $L_{\beta}$ is left zero, the first case is equivalent to $L$ being normal, and so by the Classification Theorem for homogeneous normal bands we have $L\cong \mathbb{Q} \times B_{n,1}$ for some $n\in \mathbb{N}^*$.
 Otherwise, each $L_{\beta}(l_{\alpha})$ is a single $\mathcal{L}$-class, so that $L_{\beta}(l_{\alpha})=L_{\beta}$ and  $L$ satisfies $\mathcal{D}$-covering. 

Consequently, it suffices to consider the homogeneity of linearly ordered bands satisfying $\mathcal{D}$-covering. 

\begin{proposition} Let $B=\bigcup_{\alpha\in \mathbb{Q}}B_{\alpha}$ and $B'=\bigcup_{\alpha\in \mathbb{Q}}B'_{\alpha}$ be bands satisfying $\mathcal{D}$-covering such that $B_{\alpha}\cong B_{\beta}$ and $B'_{\alpha}\cong B'_{\beta}$ for all $\alpha,\beta\in \mathbb{Q}$. 
 If $\pi \in \text{Aut }(\mathbb{Q})$ and $\theta_{\alpha}:B_{\alpha}\rightarrow B'_{\alpha\pi}$ an isomorphism for each $ \alpha$, then $\theta=\bigcup_{\alpha\in \mathbb{Q}} \theta_{\alpha}$ is an isomorphism from $B$ to $B'$. Moreover, every isomorphism can be constructed in this way.
\end{proposition} 

\begin{proof} Clearly $\theta$ is an bijection, and if $\alpha>\beta$ then, for any $e_{\alpha}\in B_{\alpha}$ and $e_{\beta}\in B_{\beta}$,  
\[ (e_{\alpha}e_{\beta})\theta=e_{\beta}\theta_{\beta}=(e_{\alpha}\theta_{\alpha})(e_{\beta}\theta_{\beta})=(e_{\alpha}\theta)(e_{\beta}\theta)
\] 
and similarly $(e_{\beta}e_{\alpha})\theta = (e_{\beta}\theta)(e_{\alpha}\theta)$. It follows that $\theta$ is a morphism, since each of the maps $\theta_{\alpha}$ are also. 
The converse follows from Proposition \ref{iso band}.  
\end{proof} 

We denote $D_{n,m}$ as the unique, up to isomorphism, linearly ordered band with structure semilattice $\mathbb{Q}$, satisfying $\mathcal{D}$-covering, and such that $B_{\alpha}\cong B_{n,m}$ for all $\alpha\in \mathbb{Q}$, where  $n,m\in \mathbb{N}^*$. 
We observe that, by uniqueness,  $D_{n,m}\cong D_{n,1}\bowtie D_{1,m}$. 

\begin{corollary} The band $D_{n,m}$ is structure-homogeneous for any $n,m\in \mathbb{N}$.  
\end{corollary} 

\begin{proof} Let $A=\bigcup_{1\leq i \leq k} A_{\alpha_i}$ and $A'=\bigcup_{1\leq i \leq k} A'_{\beta_i}$ be a finite subband of $D_{n,m}$, where $\alpha_1>\alpha_2> \cdots > \alpha_k$ and $\beta_1>\beta_2> \cdots > \beta_k$. Then $A_{\alpha_i}>A_{\alpha_j}$ if and only if $\alpha_i>\alpha_j$, and similarly for $A'$. Let $ \theta:A\rightarrow A'$ be an isomorphism, so that there exists isomorphisms $\theta_i:A_{\alpha_i}\rightarrow A'_{\beta_i}$ such that $\theta=\bigcup_{1\leq i \leq k} \theta_i$. Let $\pi \in \text{Aut }(\mathbb{Q})$ extend the unique isomorphism between $\{\alpha_1,\dots,\alpha_k\}$ and $\{\beta_1,\dots ,\beta_k\}$. By Corollary \ref{rec bands homog} we may extend each $\theta_i$ to an isomorphism $\hat{\theta}_{\alpha_i}:B_{\alpha_i}\rightarrow B_{\beta_i}$. For each $\alpha\not\in \{\alpha_1,\dots,\alpha_k\}$, take an isomorphism $\hat{\theta}_{\alpha}:B_{\alpha}\rightarrow B_{\alpha\pi}$. 
Then $\hat{\theta}=\bigcup_{\alpha\in \mathbb{Q}} \hat{\theta}_{\alpha}$ is an automorphism of $D_{n,m}$ by the previous proposition, and extends $\theta$ as required. 
\end{proof}

Now let $B=L\bowtie R$ be a homogeneous non-normal linearly ordered band not satisfying $\mathcal{D}$-covering. If $L\cong \mathbb{Q} \times B_{n,1}$ then, as shown after Proposition \ref{lin band reg}, $R$ satisfies $\mathcal{D}$-covering since $B$ is not normal. Hence $R\cong D_{1,m}$ for some $m\in \mathbb{N}^*$, and so $B\cong (\mathbb{Q} \times B_{n,1})\bowtie D_{1,m}$ by Corollary \ref{SH regular}; dually for the case $R\cong \mathbb{Q} \times B_{1,n}$.

Conversely, the bands $(\mathbb{Q} \times B_{n,1})\bowtie D_{1,m}$ and $D_{n,1}\bowtie (\mathbb{Q} \times B_{1,m})$ are structure-homogeneous (and thus homogeneous) by Corollary \ref{structure homog reg}. We thus get a complete classification of homogeneous linearly ordered bands:

\begin{theorem} \label{homog lin order}
The following are equivalent for a linearly ordered band $B$: 
\begin{enumerate}[label=(\roman*), font=\normalfont]
 \item $B$ is homogeneous;
 \item $B$ is structure-homogeneous;
 \item $B$ is isomorphic to either $\mathbb{Q}\times B_{n,m}, D_{n,m},  (\mathbb{Q} \times B_{n,1})\bowtie D_{1,m}, \text{ or } D_{n,1}\bowtie (\mathbb{Q} \times B_{1,m})$,  \\
for some $n,m\in \mathbb{N}^*$.
\end{enumerate} 
\end{theorem}

\section{The final case} \label{sec:final}

Throughout this section we let $B=\bigcup_{\alpha\in Y}B_{\alpha}$ be a non-normal band, where $Y$ is non-linear, so we may fix a three element non-chain $\alpha,\gamma,\beta$, where $\alpha\gamma=\beta$. For $e_{\alpha}\in B_{\alpha}$ and $e_{\gamma}\in B_{\gamma}$ we consider the subband $A=\langle e_{\alpha},e_{\gamma} \rangle= \{ e_{\alpha},e_{\gamma},e_{\alpha}e_{\gamma},e_{\gamma}e_{\alpha}, e_{\alpha}e_{\gamma}e_{\alpha}, e_{\gamma}e_{\alpha}e_{\gamma}  \}$, as shown in Figure  \ref{subA}.  

\begin{figure}[h]
\def\svgwidth{200pt} 
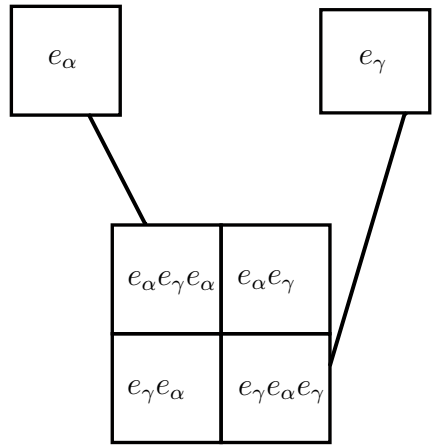 
\caption[Caption for LOF]{The subband $A$ }
\label{subA}
\end{figure}

Then $A$ is isomorphic to one of 4 bands, depending on if $A\cap B_{\beta}$ is trivial, a left zero or right zero band of size 2, or a 2 by 2 square. We will show that none of these possibilities can occur if $B$ is homogeneous.

\begin{lemma} \label{less 1} For $e_{\alpha}\in B_{\alpha}$ and $e_{\gamma}\in B_{\gamma}$ we have  $|B_{\beta}(e_{\alpha}) \cap B_{\beta}(e_{\gamma})|=1$ if and only if $B_{\beta}(e_{\alpha}) \cap B_{\beta}(e_{\gamma})\neq \emptyset$ if and only if $|A\cap B_{\beta}|=1$.
\end{lemma} 

\begin{proof} Suppose that $e_{\beta}<e_{\alpha},e_{\gamma}$. Then $e_{\alpha}e_{\gamma}\in B_{\beta}$ and, by Lemma \ref{less compatible}, $e_{\alpha}e_{\gamma}\leq e_{\beta}$, so that $e_{\beta}=e_{\alpha}e_{\gamma}$. 
Hence $B_{\beta}(e_{\alpha}) \cap B_{\beta}(e_{\gamma})=\{e_{\alpha}e_{\gamma}\}= \{e_{\gamma}e_{\alpha}\}$ and the result follows.
\end{proof}

\begin{lemma}\label{size 6} For $e_{\alpha} \in B_{\alpha}$ and $e_{\gamma}\in B_{\gamma}$ we have
\begin{enumerate}[label=(\roman*), font=\normalfont]
\item $|\langle e_{\alpha},e_{\gamma} \rangle|=6$ if and only if $\mathcal{R}(B_{\beta}(e_{\alpha}))\cap \mathcal{R}(B_{\beta}(e_{\gamma}))= \emptyset = \mathcal{L}(B_{\beta}(e_{\alpha}))\cap \mathcal{L}(B_{\beta}(e_{\gamma}))$; 
\item $|\langle e_{\alpha},e_{\gamma} \rangle| = 4$ with $e_{\alpha}e_{\gamma}e_{\alpha}=e_{\gamma}e_{\alpha}$ if and only if $\mathcal{R}(B_{\beta}(e_{\alpha}))\cap \mathcal{R}(B_{\beta}(e_{\gamma}))\neq \emptyset$ and \\
$\mathcal{L}(B_{\beta}(e_{\alpha}))\cap \mathcal{L}(B_{\beta}(e_{\gamma}))= \emptyset$. Moreover, in this case 
\begin{align*} \mathcal{R}(B_{\beta}(e_{\alpha}))\cap \mathcal{R}(B_{\beta}(e_{\gamma})) \subseteq R_{e_{\alpha}e_{\gamma}}=R_{e_{\gamma}e_{\alpha}}. 
\end{align*} 
Dually for $|\langle e_{\alpha},e_{\gamma} \rangle| = 4$ with $e_{\alpha}e_{\gamma}e_{\alpha}=e_{\alpha}e_{\gamma}$.
\end{enumerate} 
\end{lemma} 

\begin{proof} We first show that $e_{\alpha}e_{\gamma}e_{\alpha}=e_{\gamma}e_{\alpha}$  if and only if 
\[ e_{\gamma}e_{\alpha}\in \mathcal{R}(B_{\beta}(e_{\alpha}))\cap \mathcal{R}(B_{\beta}(e_{\gamma})).
\]
Since $e_{\gamma} e_{\gamma}e_{\alpha} = e_{\gamma} e_{\alpha}$ we automatically have $e_{\gamma}e_{\alpha}<_r e_{\gamma}$. Hence if $e_{\gamma}e_{\alpha}=e_{\alpha}e_{\gamma}e_{\alpha}$ then $e_{\gamma}e_{\alpha}\in B_{\beta}(e_{\alpha})$. The converse holds trivially. 

 Now suppose $e_{\beta}<_r e_{\alpha},e_{\gamma}$, so that $e_{\beta}\leq_r e_{\gamma}e_{\alpha}$, as $\leq_r$ is left compatible with multiplication.
  Hence $e_{\beta} \, \mathcal{R} \, e_{\gamma}e_{\alpha}$, and so $\mathcal{R}(B_{\beta}(e_{\alpha}))\cap \mathcal{R}(B_{\beta}(e_{\gamma}))$ is contained in $R_{e_{\alpha}e_{\gamma}}$. 
  In particular, we have shown that  $\mathcal{R}(B_{\beta}(e_{\alpha}))\cap \mathcal{R}(B_{\beta}(e_{\gamma}))$ is non-empty if and only if it contains $e_{\gamma}e_{\alpha}$. This, together with the first part of the proof gives the results. 
\end{proof}

\begin{lemma} \label{tree} Suppose there exists $\sigma>\delta>\tau$ and $e_{\sigma}>e_{\delta}$ such that $B_{\tau}(e_{\sigma})=B_{\tau}(e_{\delta})$. Then $B$ is not homogeneous. 
\end{lemma}

\begin{proof}
Suppose for contradiction that $B$ is homogeneous and $B_{\tau}(e_{\sigma})=B_{\tau}(e_{\delta})$ for some $\sigma>\delta>\tau$ and $e_{\sigma}>e_{\delta}$. Let $\sigma'>\delta'>\tau'$ in $Y$ and $e_{\sigma'}>e_{\delta'}$. Then by extending the isomorphism from $e_{\sigma}>e_{\delta}>e_{\tau}$ to $e_{\sigma'}>e_{\delta'}>e_{\tau'}$, for some $e_{\tau},e_{\tau'}$, it follows by the homogeneity of $B$ that  $B_{\tau'}(e_{\sigma'})=B_{\tau'}(e_{\delta'})$. The semilattice $Y$ is a semilinear order, since if $\eta>\{\mu,\epsilon\}>\zeta$ is a diamond in $Y$ then for any $e_{\eta}\in B_{\eta}$ with $e_{\eta}>e_{\mu},e_{\epsilon}$ we have 
\[ B_{\zeta}(e_{\mu})=B_{\zeta}(e_{\eta})=B_{\zeta}(e_{\epsilon}) 
\] 
contradicting Lemma \ref{less 1}, as $B$ is not normal. Hence $Y$ is homogeneous by  Proposition \ref{linear homog}. 
 Suppose w.l.o.g. that $B_{\tau}(e_{\sigma})$ has more than 1 $\mathcal{R}$-class. 
We claim that there exists $g_{\tau} \in \mathcal{L}(B_{\tau}(e_{\sigma}))\setminus B_{\tau}(e_{\sigma})$.
 Suppose for contradiction that no such $g_{\tau}$ exists. Then $\mathcal{R}(B_{\tau}(e_{\sigma}))=B_{\tau}$, so that for any $\nu\in Y$ with $\nu\sigma=\tau$ and $e_{\nu}\in B_{\nu}$ we would have 
\[ \mathcal{R}(B_{\tau}(e_{\sigma}))\cap \mathcal{R}(B_{\tau}(e_{\nu}))=\mathcal{R}(B_{\tau}(e_{\nu})) = B_{\tau}
\] 
by the homogeneity of  $B$. Hence $B_\tau$ has 1 $\mathcal{R}$-class by the previous pair of lemmas, a contradiction, and thus the claim holds.  

Let $g_{\tau} \in \mathcal{L}(B_{\tau}(e_{\sigma}))\setminus B_{\tau}(e_{\sigma})$. Then as $g_{\tau}<_l e_{\sigma}$ we have $g_{\tau}e_{\sigma}=g_{\tau}$,   $e_{\sigma}g_{\tau} <e_{\sigma}$ and $e_{\sigma}g_{\tau} \, \mathcal{L} \, g_{\tau}$. Letting $e_{\sigma}g_{\tau}=e_{\tau}$, then as $B_{\tau}(e_{\sigma})$ has more than one $\mathcal{R}$-class, we may pick $f_{\tau}\in B_{\tau}(e_{\sigma})$ with $f_{\tau} \, \mathcal{L} \, e_{\tau}$ and $f_{\tau}\neq e_{\tau}$. Let $A=\{e_{\tau},f_\tau , g_\tau\}$, a left zero subsemigroup of $B$.  By extending the automorphism $\theta$ of $A$ which fixes $e_{\tau}$ and swaps $f_{\tau}$ and $g_{\tau}$, to an automorphism $\bar{\theta}$ of $B$, we have $e_{\sigma}\bar{\theta}=e_{\sigma'}>e_{\tau},g_{\tau}$ and  $e_{\sigma'}f_{\tau}=e_{\tau}$. 

If $|B_{\tau}(e_{\sigma})|>2$ then there exists $x_{\tau} \not\in \{e_{\tau},f_{\tau}\}$ with $x_{\tau}\in B_{\tau}(e_{\sigma})$ and $x_{\tau}$ being $\mathcal{L}$- or $\mathcal{R}$-related to $e_{\tau}$. We may assume that $\bar{\theta}$ also extends the automorphism of $A\cup \{x_{\tau}\}$ which extends $\theta$ and fixes $x_\tau$. By the homogeneity of $B$ we have $e_{\sigma'}>e_{\tau},x_{\tau}$, so that $\sigma \sigma'>\tau$ to avoid contradicting Lemma \ref{less 1}. Then 
\[ e_{\sigma}e_{\sigma'}e_{\sigma}\cdot f_{\tau}=e_{\sigma}e_{\sigma'}f_{\tau}=e_{\sigma}e_{\tau}=e_{\tau}
\] 
so that $f_{\tau}\nless e_{\sigma}e_{\sigma'}e_{\sigma}$ and $\sigma\neq \sigma\sigma'$ (else $f_{\tau} \nless e_{\sigma}e_{\sigma'}e_{\sigma} = e_{\sigma}$). Hence $e_{\sigma}e_{\sigma'}e_{\sigma}<e_{\sigma}$ and $B_{\tau}(e_{\sigma})=B_{\tau}(e_{\sigma}e_{\sigma'}e_{\sigma})$, contradicting $f_{\tau}\not\in B_{\tau}(e_{\sigma}e_{\sigma'}e_{\sigma})$. 

It follows that $B_{\tau}(e_{\sigma})=\{e_{\tau},f_{\tau}\}$, $B_{\tau}(e_{\sigma'})=\{e_{\tau},g_{\tau}\}$, $\sigma \sigma'=\tau$ and 
\[ e_{\sigma}e_{\sigma'}e_{\sigma}=e_{\sigma'}e_{\sigma}e_{\sigma'}=e_{\tau}
\]
 by Lemma \ref{less 1}. Now extend the automorphism of $A$ which fixes $g_{\tau}$ and swaps $e_{\tau}$ and $f_{\tau}$ to an automorphism $\phi$ of $B$. Then $e_{\sigma}\phi=e_{\bar{\sigma}}>e_{\tau},f_{\tau}$, so that $\bar{\sigma}\sigma>\tau$ and $e_{\bar{\sigma}}g_{\tau}=f_{\tau}$ since $e_{\sigma}g_{\tau}=e_{\tau}$. Since $\bar{\sigma},\sigma'>\tau$ we have $\bar{\sigma}\sigma'\geq \tau$. Suppose for contradiction that  $\bar{\sigma}\sigma'>\tau$. Then we claim that $\bar{\sigma}>\{\bar{\sigma}\sigma',\bar{\sigma}\sigma\}>\tau$ forms a diamond.
Notice that $\bar{\sigma}\sigma' \neq \bar{\sigma}\sigma$, since otherwise $\bar{\sigma}\sigma=\bar{\sigma}\sigma \bar{\sigma}\sigma' = \tau$, a contradiction. If $\bar{\sigma}=\bar{\sigma}\sigma '$ then $\sigma \bar{\sigma}=\sigma \bar{\sigma}\sigma ' = \tau$ since $\sigma\sigma'=\tau$, and so $\bar{\sigma}\neq \bar{\sigma}\sigma '$, similarly $\bar{\sigma}\neq \bar{\sigma}\sigma$. Thus, as the elements are distinct, the set forms a diamond as claimed, which contradicts $Y$ being a semilinear order. Hence $\bar{\sigma}\sigma'=\tau$.  Now  $e_{\bar{\sigma}},e_{\sigma'}>e_{\tau}$, so that  $e_{\bar{\sigma}}e_{\sigma'}=e_{\sigma'}e_{\bar{\sigma}}=e_{\tau}$ and so 
  \[ e_{\bar{\sigma}}g_{\tau}= e_{\bar{\sigma}}(e_{\sigma'}g_{\tau})=e_{\tau}g_{\tau}=e_{\tau}
  \] 
as $e_{\tau} \, \mathcal{L} \, g_{\tau}$. 
However this contradicts $ {\bar{\sigma}}g_{\tau}=f_{\tau}$, and $B$ is therefore not homogeneous.
\end{proof}

\begin{lemma} \label{D class same el} Let $B$ be homogeneous and let $\alpha,\gamma,\beta\in Y$ be distinct elements such that $\alpha\gamma=\beta$. Then for any $e_{\alpha},f_{\alpha}\in B_{\alpha}$ and $e_{\gamma}\in B_{\gamma}$ such that $e_{\alpha},f_{\alpha}> e_{\alpha} e_{\gamma}e_{\alpha}, f_{\alpha} e_{\gamma} f_{\alpha}$, we have $e_{\alpha} e_{\gamma}e_{\alpha} = f_{\alpha} e_{\gamma} f_{\alpha}$.
\end{lemma} 

\begin{proof} Let $\sigma<\beta$ and choose $e_{\sigma},f_{\sigma}\in B_{\sigma}(e_{\alpha}e_{\gamma}e_{\alpha}$) such that $\langle e_{\sigma},f_{\sigma} \rangle$ is isomorphic to $ \langle e_{\alpha} e_{\gamma}e_{\alpha}, f_{\alpha} e_{\gamma} f_{\alpha} \rangle$, noting that such elements exist by Corollary \ref{basics} $(\mathrm{iv}$). Extend the isomorphism from $\langle e_{\alpha},f_{\alpha},e_{\sigma},f_{\sigma} \rangle$ to $\langle e_{\alpha},f_{\alpha}, e_{\alpha} e_{\gamma}e_{\alpha}, f_{\alpha} e_{\gamma} f_{\alpha} \rangle$ which maps generators in order, to an automorphism of $B$. Then it follows that there exists $\tau\in Y$ and $e_{\tau}\in B_{\tau}$ (as the image of $e_{\alpha}e_\gamma e_{\alpha}$) such that $\alpha>\tau>\beta$ and 
\[ \{e_{\alpha},f_{\alpha}\}> e_{\tau} > \{  e_{\alpha} e_{\gamma}e_{\alpha}, f_{\alpha} e_{\gamma} f_{\alpha}\}.
\] 
Then \[ e_{\alpha} e_{\gamma} e_{\alpha} = e_{\tau}(e_{\alpha} e_{\gamma} e_{\alpha})e_{\tau} = (e_{\tau}e_{\alpha}) e_{\gamma} (e_{\alpha}e_{\tau})=e_{\tau}e_{\gamma}e_{\tau},
\]  
and similarly $f_{\alpha} e_{\gamma} f_{\alpha} =e_{\tau}e_{\gamma}e_{\tau}$, and the result follows.
%
%
\end{proof}

\begin{lemma} \label{case 1} If $B$ is homogeneous, $e_{\alpha}\in B_{\alpha}$ and $e_{\gamma}\in B_{\gamma}$ then $|\langle e_{\alpha},e_{\gamma} \rangle|\neq 6$.
\end{lemma} 

\begin{proof} Suppose for contradiction that $|\langle e_{\alpha},e_{\gamma} \rangle|=6$ and let $e_{\beta}\in B_{\beta}(e_{\alpha})$. Note that if $D$ is a rectangular band and $x,y,z\in D$ then $xyz=(xy)(zxz)=x(yz)x z = xz$. Hence 
\begin{align*}
& e_{\beta}e_{\gamma}e_{\beta}=e_{\beta}(e_{\alpha}e_{\gamma}e_{\alpha})e_{\beta}=e_{\beta}\\
& e_{\gamma}e_{\beta}e_{\gamma}=(e_{\gamma}e_{\alpha})e_{\beta}(e_{\alpha} e_{\gamma})=(e_{\gamma}e_{\alpha})(e_{\alpha}e_{\gamma})=e_{\gamma}e_{\alpha}e_{\gamma}.
\end{align*}
By Lemma \ref{size 6} $(\mathrm{i})$, $e_{\beta}$ is not $\mathcal{L}$ or $\mathcal{R}$-related to $e_{\gamma}e_{\alpha}e_{\gamma}$, so 
\[ \langle e_{\gamma},e_{\beta} \rangle = \{e_{\gamma},e_{\beta},e_{\gamma}e_{\beta},e_{\beta}e_{\gamma},e_{\gamma}e_{\alpha}e_{\gamma} \}
\]
contains no repetitions. Hence for any $e_{\beta},f_{\beta}<e_{\alpha}$ we have $\langle e_{\gamma},e_{\beta} \rangle \cong \langle e_{\gamma},f_{\beta} \rangle$. In particular, the map fixing $e_{\gamma}$ and swapping some $e_{\beta}\in B_{\beta}(e_{\alpha})\setminus \{e_{\alpha}e_{\gamma}e_{\alpha}\}$ with $e_{\alpha}e_{\gamma}e_{\alpha}$ is an isomorphism, which can extended to $\theta\in \text{Aut}(B)$. 
Then $e_{\alpha}>e_{\beta},e_{\alpha}e_{\gamma}e_{\alpha}$ gives $e_{\alpha}\theta=e_{\alpha'}>e_{\alpha}e_{\gamma}e_{\alpha},e_{\beta}$, so that $\alpha\alpha'>\beta$ by Lemma \ref{less 1}. Moreover, $(e_{\alpha}e_{\gamma}e_{\alpha})\theta= e_{\alpha'}e_{\gamma}e_{\alpha'}=e_{\beta}$, so
\[ (e_{\alpha}e_{\alpha'}e_{\alpha})e_{\gamma}(e_{\alpha}e_{\alpha'}e_{\alpha})=(e_{\alpha}e_{\alpha'})(e_{\alpha}e_{\gamma}e_{\alpha})(e_{\alpha'}e_{\alpha})=e_{\alpha}e_{\gamma}e_{\alpha}
\] 
by Lemma \ref{less compatible}, and similarly $(e_{\alpha'}e_{\alpha}e_{\alpha'})e_{\gamma}(e_{\alpha'}e_{\alpha}e_{\alpha'})= e_{\alpha'}e_{\gamma}e_{\alpha'}$. 
Hence  
\[ \{e_{\alpha}e_{\alpha'}e_{\alpha}, e_{\alpha'}e_{\alpha}e_{\alpha'}\}>\{e_{\alpha}e_{\gamma}e_{\alpha},e_{\beta}\},
\] 
and $\{e_{\alpha}e_{\gamma}e_{\alpha},e_{\beta}\}=\{(e_{\alpha}e_{\alpha'}e_{\alpha})e_{\gamma}(e_{\alpha}e_{\alpha'}e_{\alpha}),(e_{\alpha'}e_{\alpha}e_{\alpha'})e_{\gamma}(e_{\alpha'}e_{\alpha}e_{\alpha'})\}$. Since $(\alpha\alpha')\gamma=\beta$ with $\alpha\alpha'\neq \gamma$ we have $e_{\alpha}e_{\gamma}e_{\alpha} =e_{\beta}$ by Lemma \ref{D class same el}, a contradiction. 
\end{proof} 

\begin{lemma}  If $B$ is homogeneous, $e_{\alpha}\in B_{\alpha}$ and $e_{\gamma}\in B_{\gamma}$ then  $|\langle e_{\alpha},e_{\gamma} \rangle|\neq 4$.
\end{lemma} 

\begin{proof} Suppose for contradiction that $|\langle e_{\alpha},e_{\gamma} \rangle|= 4$, and assume w.l.o.g. that $e_{\alpha}e_{\gamma}e_{\alpha}=e_{\gamma}e_{\alpha}$, so $e_{\gamma}e_{\alpha}e_{\gamma}=e_{\alpha}e_{\gamma}$. By Lemma \ref{size 6} $(\mathrm{ii})$ we have $\mathcal{L}(B_{\beta}(e_{\alpha}))\cap \mathcal{L}(B_{\beta}(e_{\gamma}))= \emptyset$ and $\mathcal{R}(B_{\beta}(e_{\alpha}))\cap \mathcal{R}(B_{\beta}(e_{\gamma}))\subseteq R_{e_{\gamma}e_{\alpha}}=R_{e_{\alpha}e_{\gamma}}$. Suppose $B_{\beta}(e_{\alpha})$ has more than 1 $\mathcal{L}$-class, so there exists $e_{\beta}\in  B_{\beta}(e_{\alpha})$ such that $e_{\beta} \, \mathcal{R} \, e_{\gamma}e_{\alpha}$ but $e_{\beta}\neq e_{\gamma}e_{\alpha}$, noting that $e_{\beta}\neq e_{\alpha}e_{\gamma}$ as $|\langle e_{\alpha},e_{\gamma} \rangle|\neq 3$ (see Figure \ref{rec pic}).

\begin{figure}[h]
\centering
\def\svgwidth{130pt} 
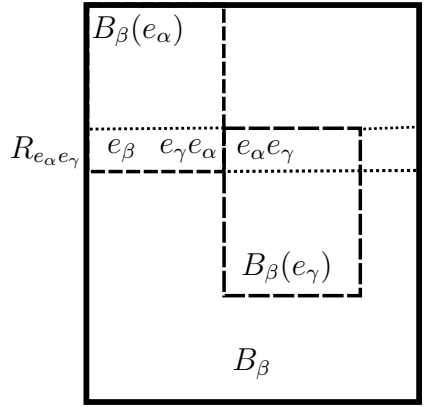 
\caption[Caption for LOF]{The rectangular band $B_{\beta}$ }
\label{rec pic} 
\end{figure}

 Since $e_{\beta},e_{\gamma}e_{\alpha}<_r e_{\gamma}$ and $e_{\beta}e_{\gamma}=e_{\beta}e_{\alpha}e_{\gamma}=e_{\alpha}e_{\gamma}$ we have that the subband
\[ C=\langle e_{\gamma},e_{\beta},e_{\gamma}e_{\alpha} \rangle =\{e_{\gamma},e_{\beta},e_{\gamma}e_{\alpha}, e_{\alpha}e_{\gamma}\}
\] 
contains no repetitions. We may thus extend the automorphism of $C$ which fixes $e_{\gamma}$ and $e_{\alpha}e_{\gamma}$ and swaps $e_{\beta}$ and $e_{\gamma}e_{\alpha}$, to an automorphism $\theta$ of $B$. Then  $e_{\alpha'}=e_{\alpha}\theta>e_{\beta},e_{\gamma}e_{\alpha}$, so that $\alpha\alpha'>\beta$, and $(e_{\gamma}e_{\alpha})\theta=e_{\gamma}e_{\alpha'}=e_{\beta}$. Following the proof of the previous lemma we obtain $\{e_{\alpha}e_{\alpha'}e_{\alpha}, e_{\alpha'}e_{\alpha}e_{\alpha'}\}> \{(e_{\alpha}e_{\alpha'}e_{\alpha})e_{\gamma}(e_{\alpha}e_{\alpha'}e_{\alpha}),(e_{\alpha'}e_{\alpha}e_{\alpha'})e_{\gamma}(e_{\alpha'}e_{\alpha}e_{\alpha'})\}$ and so $e_{\alpha}e_{\gamma}e_{\alpha} =e_{\gamma}e_{\alpha}= e_{\beta}$, a contradiction. Hence $B_{\beta}(e_{\alpha})$ is a left zero band. 

Let $\tau\in Y$ and $e_{\tau}\in B_{\tau}$ be such that $\beta<\tau<\alpha$, so $\tau\gamma=\beta$, and $e_{\tau}<e_{\alpha}$, so that $B_{\beta}(e_{\tau})\subseteq B_{\beta}(e_{\alpha})$. If $e_{\tau}\ngtr e_{\gamma}e_{\alpha}$, then $\mathcal{R}(B_{\beta}(e_{\tau}))\cap \mathcal{R}(B_{\beta}(e_{\gamma}))=\emptyset$ as $B_{\beta}(e_{\tau})$ is also left zero by Corollary \ref{basics} $(\mathrm{iv})$. Hence $|\langle e_{\tau},e_{\gamma} \rangle|=6$ by Lemma \ref{size 6}, a contradiction, and thus $e_{\tau}>e_{\gamma}e_{\alpha}$. Suppose for contradiction that $e_{\tau}\ngtr f_{\beta}$, for some $f_{\beta}\in B_{\beta}(e_{\alpha})$. Then by extending the automorphism of $\{e_{\alpha},f_\beta, e_\gamma e_\alpha \}$ which fixes $e_{\alpha}$ and swaps $f_{\beta}$ with $e_{\gamma}e_{\alpha}$ to an automorphism of $B$, we obtain some $e_{\sigma}$ (the image of $e_{\tau}$) such that $\alpha>\sigma>\beta$ and $e_{\sigma}\ngtr e_{\gamma}e_{\alpha}$, a contradiction. Thus $B_{\beta}(e_{\alpha})=B_{\beta}(e_{\tau})$, and so $B$ is not homogeneous by Lemma \ref{tree}. 
\end{proof} 

\begin{lemma} If $B$ is a non-normal homogeneous band then it is linearly ordered. 
\end{lemma} 

\begin{proof} Let $B=\bigcup_{\alpha}B_{\alpha}$ be a non-normal homogeneous band. Suppose for contradiction that $Y$ contains a three element non-chain $\alpha,\gamma,\beta$, where $\alpha\gamma=\beta$. It follows from the preceding lemmas that for any $e_{\alpha}\in B_{\alpha}$ and $e_{\gamma}\in B_{\gamma}$ we have $\langle e_{\alpha},e_{\gamma}\rangle =\{e_{\alpha},e_{\gamma},e_{\alpha}e_{\gamma}\}$. Hence $B_{\beta}(e_{\alpha})\cap B_{\beta}(e_{\gamma})=\{e_{\alpha}e_{\gamma}\}$ and $e_{\alpha}e_{\gamma}=e_{\gamma}e_{\alpha}=e_{\alpha}e_{\gamma}e_{\alpha}=e_{\gamma}e_{\alpha}e_{\gamma}$. 

For any $\alpha>\delta>\beta$ and $e_{\alpha}>e_{\delta}$ we have  $e_{\delta}>e_{\alpha}e_{\gamma}$. Indeed, if $e_{\delta}\ngtr e_{\alpha}e_{\gamma}$ then as 
\[ B_{\beta}(e_{\delta})\cap B_{\beta}(e_{\gamma})\subseteq B_{\beta}(e_{\alpha})\cap B_{\beta}(e_{\gamma}) = \{e_{\alpha}e_{\gamma}\}
\] 
we have $ B_{\beta}(e_{\delta})\cap B_{\beta}(e_{\gamma})= \emptyset$ and so $|\langle e_{\delta},e_{\gamma} \rangle|>3$, a contradiction. For any $e_{\beta}\in B_{\beta}(e_{\alpha})\setminus B_{\beta}(e_{\delta})$, extend the automorphism of $\{e_\alpha, e_{\alpha} e_\gamma , e_\beta \}$ which fixes $e_{\alpha}$ and swaps $e_{\alpha}e_{\gamma}$ and $e_\beta$ to an automorphism $\theta$ of $B$. Letting $e_{\delta}\theta=e_{\tau}$, then $e_{\tau}<e_{\alpha}=e_{\alpha}\theta$ and $e_{\tau}\not> e_{\alpha}e_{\gamma}$, a contradiction. Thus $B_{\beta}(e_{\alpha})=B_{\beta}(e_{\delta})$, contradicting Lemma \ref{tree}.
\end{proof}

This, together with the Classification Theorem for homogeneous normal bands and Theorem \ref{homog lin order} gives: 

 \begin{theorem}[Classification Theorem for homogeneous bands] A band is homogeneous if and only if isomorphic to either a homogeneous normal band or a homogeneous linearly ordered band. 
 \end{theorem} 

An immediate consequence is that the structure semilattice of a homogeneous band is homogeneous.  \\

\noindent {\bf Open Problem.} Prove directly that the homogeneity of a band is inherited by its structure semilattice.  \\

By Proposition \ref{SH normal} and Theorem \ref{homog lin order}  we achieve a complete list of structure-homogeneous bands: 

\begin{theorem}[Classification Theorem for structure-homogeneous bands] A band is \,  structure-homogeneous if and only if isomorphic to either a homogeneous linearly ordered band or the band $Y \times B_{n,m}$, for some homogeneous semilattice $Y$ and $n,m\in \mathbb{N}^*$. 
\end{theorem}

\end{document}

%% file: pic4.pdf_tex
\begingroup%
  \makeatletter%
  \providecommand\color[2][]{%
    \errmessage{(Inkscape) Color is used for the text in Inkscape, but the package 'color.sty' is not loaded}%
    \renewcommand\color[2][]{}%
  }%
  \providecommand\transparent[1]{%
    \errmessage{(Inkscape) Transparency is used (non-zero) for the text in Inkscape, but the package 'transparent.sty' is not loaded}%
    \renewcommand\transparent[1]{}%
  }%
  \providecommand\rotatebox[2]{#2}%
  \ifx\svgwidth\undefined%
    \setlength{\unitlength}{595.27559055bp}%
    \ifx\svgscale\undefined%
      \relax%
    \else%
      \setlength{\unitlength}{\unitlength * \real{\svgscale}}%
    \fi%
  \else%
    \setlength{\unitlength}{\svgwidth}%
  \fi%
  \global\let\svgwidth\undefined%
  \global\let\svgscale\undefined%
  \makeatother%
  \begin{picture}(1,1.41428571)%
    \put(0,0){\includegraphics[width=\unitlength,page=1]{pic4.pdf}}%
    \put(0.54381327,0.12851333){\color[rgb]{0,0,0}\makebox(0,0)[lb]{\smash{$\kappa$}}}%
    \put(0.03108161,0.49521044){\color[rgb]{0,0,0}\makebox(0,0)[lb]{\smash{$\rho$}}}%
    \put(0.57190604,0.27907361){\color[rgb]{0,0,0}\makebox(0,0)[lb]{\smash{$\beta$}}}%
    \put(0.34349838,0.42822617){\color[rgb]{0,0,0}\makebox(0,0)[lb]{\smash{$\tau$}}}%
    \put(0.34773725,0.62435131){\color[rgb]{0,0,0}\makebox(0,0)[lb]{\smash{$\sigma$}}}%
    \put(0.54081127,0.88752178){\color[rgb]{0,0,0}\makebox(0,0)[lb]{\smash{$\alpha$}}}%
    \put(0.42716167,1.25681841){\color[rgb]{0,0,0}\makebox(0,0)[lb]{\smash{$\delta$ }}}%
    \put(0.94715557,0.85615879){\color[rgb]{0,0,0}\makebox(0,0)[lb]{\smash{$\gamma$}}}%
    \put(0,0){\includegraphics[width=\unitlength,page=2]{pic4.pdf}}%
    \put(0.5414059,0.52796674){\color[rgb]{0,0,0}\makebox(0,0)[lb]{\smash{$\sigma\alpha$}}}%
  \end{picture}%
\endgroup%

%% file: pic3.pdf_tex
\begingroup%
  \makeatletter%
  \providecommand\color[2][]{%
    \errmessage{(Inkscape) Color is used for the text in Inkscape, but the package 'color.sty' is not loaded}%
    \renewcommand\color[2][]{}%
  }%
  \providecommand\transparent[1]{%
    \errmessage{(Inkscape) Transparency is used (non-zero) for the text in Inkscape, but the package 'transparent.sty' is not loaded}%
    \renewcommand\transparent[1]{}%
  }%
  \providecommand\rotatebox[2]{#2}%
  \ifx\svgwidth\undefined%
    \setlength{\unitlength}{595.27559055bp}%
    \ifx\svgscale\undefined%
      \relax%
    \else%
      \setlength{\unitlength}{\unitlength * \real{\svgscale}}%
    \fi%
  \else%
    \setlength{\unitlength}{\svgwidth}%
  \fi%
  \global\let\svgwidth\undefined%
  \global\let\svgscale\undefined%
  \makeatother%
  \begin{picture}(1,1.41428571)%
    \put(0,0){\includegraphics[width=\unitlength,page=1]{pic3.pdf}}%
    \put(0.41008492,0.27335432){\color[rgb]{0,0,0}\makebox(0,0)[lb]{\smash{$Z_1$ }}}%
    \put(0,0){\includegraphics[width=\unitlength,page=2]{pic3.pdf}}%
    \put(0.54512038,0.59498223){\color[rgb]{0,0,0}\makebox(0,0)[lb]{\smash{$z$}}}%
    \put(0,0){\includegraphics[width=\unitlength,page=3]{pic3.pdf}}%
    \put(0.75760271,1.03950832){\color[rgb]{0,0,0}\makebox(0,0)[lb]{\smash{$a_3$}}}%
    \put(0.06652042,1.28289373){\color[rgb]{0,0,0}\makebox(0,0)[lb]{\smash{}}}%
    \put(0,0){\includegraphics[width=\unitlength,page=4]{pic3.pdf}}%
    \put(0.00373125,1.26620137){\color[rgb]{0,0,0}\makebox(0,0)[lb]{\smash{$Y_1=Z_0$}}}%
    \put(0.60041216,1.25409073){\color[rgb]{0,0,0}\makebox(0,0)[lb]{\smash{$A_z$ }}}%
    \put(0,0){\includegraphics[width=\unitlength,page=5]{pic3.pdf}}%
  \end{picture}%
\endgroup%

%% file: pic1.pdf_tex
\begingroup%
  \makeatletter%
  \providecommand\color[2][]{%
    \errmessage{(Inkscape) Color is used for the text in Inkscape, but the package 'color.sty' is not loaded}%
    \renewcommand\color[2][]{}%
  }%
  \providecommand\transparent[1]{%
    \errmessage{(Inkscape) Transparency is used (non-zero) for the text in Inkscape, but the package 'transparent.sty' is not loaded}%
    \renewcommand\transparent[1]{}%
  }%
  \providecommand\rotatebox[2]{#2}%
  \ifx\svgwidth\undefined%
    \setlength{\unitlength}{376.88135986bp}%
    \ifx\svgscale\undefined%
      \relax%
    \else%
      \setlength{\unitlength}{\unitlength * \real{\svgscale}}%
    \fi%
  \else%
    \setlength{\unitlength}{\svgwidth}%
  \fi%
  \global\let\svgwidth\undefined%
  \global\let\svgscale\undefined%
  \makeatother%
  \begin{picture}(1,0.83189829)%
    \put(0,0){\includegraphics[width=\unitlength,page=1]{pic1.pdf}}%
    \put(0.07203438,0.72023996){\color[rgb]{0,0,0}\makebox(0,0)[lb]{\smash{$e_{\alpha}$ }}}%
    \put(0.66028766,0.71902233){\color[rgb]{0,0,0}\makebox(0,0)[lb]{\smash{$e_{\gamma}$ }}}%
    \put(0.22258514,0.30046235){\color[rgb]{0,0,0}\makebox(0,0)[lb]{\smash{$e_{\alpha}e_{\gamma}e_{\alpha}$ }}}%
    \put(0.22238571,0.09501844){\color[rgb]{0,0,0}\makebox(0,0)[lb]{\smash{$e_{\gamma}e_{\alpha}$}}}%
    \put(0.43182854,0.0937067){\color[rgb]{0,0,0}\makebox(0,0)[lb]{\smash{$e_{\gamma}e_{\alpha}e_{\gamma}$}}}%
    \put(0.06745245,0.15044634){\color[rgb]{0,0,0}\makebox(0,0)[lt]{\begin{minipage}{0.80358751\unitlength}\raggedright \end{minipage}}}%
    \put(0.21149166,0.36423095){\color[rgb]{0,0,0}\makebox(0,0)[lt]{\begin{minipage}{0.58222185\unitlength}\raggedright \end{minipage}}}%
    \put(0.43048217,0.3026487){\color[rgb]{0,0,0}\makebox(0,0)[lb]{\smash{$e_\alpha e_{\gamma}$}}}%
  \end{picture}%
\endgroup%

%% file: pic2.pdf_tex
\begingroup%
  \makeatletter%
  \providecommand\color[2][]{%
    \errmessage{(Inkscape) Color is used for the text in Inkscape, but the package 'color.sty' is not loaded}%
    \renewcommand\color[2][]{}%
  }%
  \providecommand\transparent[1]{%
    \errmessage{(Inkscape) Transparency is used (non-zero) for the text in Inkscape, but the package 'transparent.sty' is not loaded}%
    \renewcommand\transparent[1]{}%
  }%
  \providecommand\rotatebox[2]{#2}%
  \ifx\svgwidth\undefined%
    \setlength{\unitlength}{595.27559055bp}%
    \ifx\svgscale\undefined%
      \relax%
    \else%
      \setlength{\unitlength}{\unitlength * \real{\svgscale}}%
    \fi%
  \else%
    \setlength{\unitlength}{\svgwidth}%
  \fi%
  \global\let\svgwidth\undefined%
  \global\let\svgscale\undefined%
  \makeatother%
  \begin{picture}(1,1.41428571)%
    \put(0,0){\includegraphics[width=\unitlength,page=1]{pic2.pdf}}%
    \put(0.07819558,0.81074925){\color[rgb]{0,0,0}\makebox(0,0)[lb]{\smash{$e_{\beta}$ }}}%
    \put(0.45468664,0.80839864){\color[rgb]{0,0,0}\makebox(0,0)[lb]{\smash{$e_{\alpha}e_{\gamma}$}}}%
    \put(0.23034565,0.81020161){\color[rgb]{0,0,0}\makebox(0,0)[lb]{\smash{$e_{\gamma}e_{\alpha}$}}}%
    \put(-0.20465397,0.7982542){\color[rgb]{0,0,0}\makebox(0,0)[lb]{\smash{$R_{e_{\alpha}e_{\gamma}}$ }}}%
    \put(0.03306474,1.14772785){\color[rgb]{0,0,0}\makebox(0,0)[lb]{\smash{$B_{\beta}(e_{\alpha})$ }}}%
    \put(0.46870464,0.45745797){\color[rgb]{0,0,0}\makebox(0,0)[lb]{\smash{$B_{\beta}(e_{\gamma})$ }}}%
    \put(0.44288821,0.17867828){\color[rgb]{0,0,0}\makebox(0,0)[lb]{\smash{$B_{\beta}$ }}}%
    \put(0,0){\includegraphics[width=\unitlength,page=2]{pic2.pdf}}%
  \end{picture}%
\endgroup%